\documentclass[12pt, reqno]{amsart}
\usepackage{amsthm, amsopn, times}
\usepackage{amssymb}
\usepackage{amsmath}
\usepackage{mathtools}
\usepackage{bbm}
\usepackage{extarrows}
\usepackage{wasysym}

\author [Bhattacharyya, Bhowmik]{Tirthankar Bhattacharyya, Mainak Bhowmik}
\address{Department of Mathematics, 	Indian Institute of Science, 
	Bangalore 560012, India}

\author[Sau]{Haripada Sau}
\address{Department of Mathematics, Indian Institute of Science Education and Research, Pune 411008, India.}

\email{tirtha@iisc.ac.in;  mainakb@iisc.ac.in; hsau@iiserpune.ac.in}

\usepackage[x11names]{xcolor}
\usepackage[linkcolor=Blue2,citecolor=red,colorlinks]{hyperref}
\usepackage[capitalize,nameinlink]{cleveref}
\usepackage[numbers,sort&compress]{natbib}
\usepackage{color, bm, amscd, tikz-cd}

\newcommand{\cA}{{\mathcal A}}
\newcommand{\cB}{{\mathcal B}}

\newcommand{\cE}{{\mathcal E}}

\newcommand{\cH}{{\mathcal H}}

\newcommand{\cJ}{{\mathcal J}}
\newcommand{\cK}{{\mathcal K}}

\newcommand{\cN}{{\mathcal N}}

\newcommand{\cP}{{\mathcal P}}

\newcommand{\bT}{{\mathbb{T}}}
\newcommand{\bB}{{\mathbb{B}}}
\newcommand{\bC}{{\mathbb{C}}}
\newcommand{\bD}{{\mathbb{D}}}

\newtheorem{thm}{Theorem}[section]

\newtheorem{corollary}[thm]{Corollary}
\newtheorem{lemma}[thm]{Lemma}

\newtheorem{proposition}[thm]{Proposition}

\theoremstyle{definition}

\newtheorem{theorem}{Theorem}
\newtheorem{definition}[thm]{Definition}
\newtheorem{remark}[thm]{Remark}

\newtheorem{example}[thm]{Example}

\newcommand*{\defeq}{\mathrel{\vcenter{\baselineskip0.5ex \lineskiplimit0pt
                     \hbox{\scriptsize.}\hbox{\scriptsize.}}}
                     =}

\numberwithin{equation}{section}

\def\textmatrix#1&#2\\#3&#4\\{\bigl({#1 \atop #3}\ {#2 \atop #4}\bigr)}
\def\dispmatrix#1&#2\\#3&#4\\{\left({#1 \atop #3}\ {#2 \atop #4}\right)}
\numberwithin{equation}{section}

\def\textmatrix#1&#2\\#3&#4\\{\bigl({#1 \atop #3}\ {#2 \atop #4}\bigr)}
\def\dispmatrix#1&#2\\#3&#4\\{\left({#1 \atop #3}\ {#2 \atop #4}\right)}

\begin{document}
\setcounter{page}{1}

\title[Hankel operators and Projective Hilbert modules]{Hankel operators and Projective Hilbert modules on quotients of bounded symmetric domains}

\maketitle

\begin{abstract}  Consider a bounded symmetric domain $\Omega$ with a finite pseudo-reflection group acting on it as a subgroup of the group of automorphisms. This gives rise to quotient domains by means of basic polynomials $\theta$ which by virtue of being proper maps map the \v Silov boundary of $\Omega$ to the \v Silov boundary of $\theta(\Omega)$. Thus, the natural measure on the \v Silov boundary of $\Omega$ can be pushed forward. This gives rise to Hardy spaces on the quotient domain. 
  
  The study of Hankel operators on the Hardy spaces of the quotient domains is introduced. The use of the weak product space shows that an analogue of Hartman's theorem holds for the small Hankel operator. Nehari's theorem fails for the big Hankel operator and this has the consequence that when the domain $\Omega$ is the polydisc $\mathbb D^d$, the {\em Hardy space} is not a projective object in the category of all Hilbert modules over the algebra $\mathcal A (\theta(\mathbb D^d))$ of functions which are holomorphic in the quotient domain and continuous on the  closure $\overline {\theta(\mathbb D^d)}$. It is not a projective object in the category of cramped Hilbert modules either. Indeed, no projective object is known in these two categories. On the other hand, every normal Hilbert module over the algebra of continuous functions on the \v Silov boundary, treated as a Hilbert module over the algebra $\mathcal A (\theta(\mathbb D^d))$,  is projective. \end{abstract}

\renewcommand{\thefootnote}{\fnsymbol{footnote}} 

\tableofcontents

\section{Introduction and Background} \label{prelims}
The classical theorem of Nehari \cite{Nehari} which states that a $\varphi \in H^2(\bD)$ defines a bounded Hankel operator on $H^2(\bD)$ if and only if $\varphi \in (H^1(\bD))^*$ has intrigued mathematicians over generations, see \cite{Car-Cla1, Fer, Fe-Sa, Hartman, Russo} and references therein. Attempts to generalize Nehari's result to more than one variable have been successful in certain domains, see \cite{CRW, Krantz-Li} and have met with serious road blocks elsewhere, see \cite{LPPW}.

A class of domains which have emerged as important domains both function theoretically and operator theoretically arises from a group action on bounded symmetric domains. 
\begin{definition} \label{D: Pseudo-reflection}
Let $d >1$ be a fixed positive integer. A {\em pseudo-reflection} is a linear map $\sigma: \bC^d \to \bC^d$ such that  rank of $(I_d - \sigma) =1$ and $\sigma^n = I_d$ for some positive integer $n$. A group $G$ generated by pseudo-reflections is known as a {\em pseudo-reflection group}. The linear map $\sigma$ is called a reflection if it is of order $2$.
\end{definition}
A domain $\Omega$ in $\bC^d$ is said to be $G$-invariant if, under the action $\sigma \cdot \boldsymbol{z} = \sigma^{-1}(\boldsymbol{z})$ of $G$ in $\bC^d$, the domain remains invariant. Clearly, $G$ is then a subgroup of Aut$(\Omega)$, the automorphism group of the domain $\Omega$. The action on a $G$-invariant domain $\Omega$ induces the action $(\sigma \cdot f)(\boldsymbol{z}) = f(\sigma^{-1} \cdot \boldsymbol{z}) = f(\sigma(\boldsymbol{z}))$ of $G$ on complex-valued functions defined on $\Omega$. A $G$-invariant function $f$ is a function on a $G$-invariant domain $\Omega$ such that $\sigma \cdot f= f$ for every $\sigma \in G$. For example, a symmetric polynomial when $\Omega = \bD^d$ and $G$ is the permutation group $S_d$.

Landmark works of Chevalley \cite{Chevalley}, Shephard and Todd \cite{She-Todd} as well as Bedford and Dadok \cite{Bedford} prove that for a finite linear group $G$ and a $G$-invariant domain $\Omega$, the quotient $\Omega/G$ is the image $\theta(\Omega)$ under a polynomial map (known as the {\em basic polynomial map associated to $G$}) 
$$
\theta(\boldsymbol{z})= (\theta_1(\boldsymbol{z}), \dots, \theta_d(\boldsymbol{z})) \text{ for } \boldsymbol{z} \in \Omega
$$
if and only if $G$ is a pseudo-reflection group. The map $\theta$ is not unique. However, the $\theta_j$ can be chosen to be algebraically independent homogeneous polynomials.

The proper map $\theta$ extends as a proper map of the same multiplicity from a neighbourhood of $\overline{\Omega}$ to a neighbourhood of $\overline{\theta(\Omega)}$ and the \v Silov boundary $\partial \Omega$ of $\Omega$ with respect to the uniform algebra  $\mathcal A (\theta(\Omega))$ of functions which are holomorphic in the domain $\Omega$ and continuous on $\overline \Omega$ is the same as $\theta^{-1}(\partial \theta(\Omega))$ where $\partial\theta(\Omega)$ is the \v Silov boundary of $\theta(\Omega)$ \cite{Kos-Zow}.  This implies that $\partial \theta(\Omega)= \theta(\partial \Omega)$). In case $\overline{\Omega}$ is polynomially convex, $\overline{\theta(\Omega)}$ is a polynomially convex compact set, see \cite[Theorem 1.6.24]{Stout}.

A long list of complex analysts and operator theorists have invested in quotient domains when $\Omega$ is the polydisc ($\bD^d$) or the Euclidean unit ball ($\bB_d$) or the bounded symmetric domains of type II. We note the contributions made in \cite{Rudin-IUMJ, Agler-Young, A-Y-W,B-P-R, Tirtha-Hari-JFA, E-K-Z, Costara-JLMS, Bhattacharyya-IUMJ, Ball-Sau, MRZ, BDGR-adv, Pal-Roy} and in particular the recent preprint \cite{GSR} which has a long list of references.

One of the key results that make our work possible is the {\em analytic Chevalley-Shepard-Todd} Theorem obtained in \cite{BDGR-adv} which states that for a pseudo-reflection group $G$ acting on an invariant domain $\Omega$ along with its basic polynomials $\theta$ as above, any holomorphic function $f : \Omega \rightarrow \mathbb C$ invariant under $G$ can be written as $f = g \circ \theta$ for a holomorphic function $g : \theta(\Omega) \rightarrow \mathbb C$.  

The Hankel operators are well understood on the disc \cite{Nehari, Peller-book}, on the Euclidean ball \cite{CRW} and on more general smoothly bounded strongly pseudoconvex domains \cite{Krantz-Li}. Also, they have been studied on the polydisc \cite{Co-Sa, Cotlar-Sadosky-IEOT, Fe-Sa}. They are connected with the Nevanlinna-Pick interpolation \cite{Co-Sa, Peller-book}, Carath\'eodory-Fej\'er interpolation \cite{Peller-book}, similarity problems \cite{Davidson-Paulsen} and projectivity of Hardy modules \cite{Car-Cla1, Ferguson-PAMS, Guo-Studia} to name a few. Characterizing the boundedness and compactness of Hankel operators automatically relates function theory with various topics in harmonic analysis.This is why we are going to study Hankel operators on the quotient domains. It is challenging because of their intrinsic geometry as they are mostly non-smooth domains. 

This paper has two parts. Part A consists of function theory whose algebraic consequences are elucidated in Part B. Among the various approaches to the theory of bounded operators on Hilbert spaces, none has been more intriguing than the one through Hilbert modules as can be seen from \cite{D-P_Cite, JS}. \begin{definition}
A Hilbert module $\mathcal{H}$ over a function algebra $\mathcal{A}$ is a Hilbert space $\mathcal{H}$ which is also a module over $\mathcal{A}$ such that $(a,h) \rightarrow a\cdot h$ is a continuous function. If moreover,  the linear operator $T_a:\mathcal{H}\to \mathcal{H}$ defined by $T_a(h):= a\cdot h$ satisfies 
\begin{align}\label{contractive-HM}
\|T_a(h)\| \leq \|a\|\|h\| \quad \text{for all}\,\, a\in \mathcal{A} \quad \text{and} \quad h\in \mathcal{H}, 
\end{align}
then $\mathcal{H}$ is said to be a contractive Hilbert module. 
\end{definition}
 We shall concentrate on contractive Hilbert modules over $\cA(\theta(\bD^d))$ where $\theta$ is a basic polynomial corresponding to the action of a pseudo-reflection group $G$ on $\bD^d$. The actions of the co-ordinate functions $p_j$ for $j=1, \dots, d$ on the Hilbert module $\mathcal{H}$ are denoted by $T_{p_j}$ for $j=1,\dots, d$. We denote the category of Hilbert modules over $\cA(\theta(\bD^d))$ by $\mathfrak{H}$. 
\begin{definition}
Let $\mathcal{H}$ and $\mathcal{K}$ be two Hilbert modules over $\cA(\theta(\bD^d))$. A module map between $\mathcal{H}$ and $\mathcal{K}$ is a bounded linear map $X: \mathcal{H} \to \mathcal{K}$ such that 
$$X(f\cdot h)=f \cdot X(h)\quad \text{for all}\,\, h\in \mathcal{H}, f\in \cA(\theta(\bD^d)).$$ 
We say that $\mathcal{H}$ and $\mathcal{K}$ are \textit{similar} if $X$ is invertible. A Hilbert module over $\cA(\theta(\bD^d))$ is said to be \textit{cramped} if it is similar to a contractive Hilbert module. We denote the category of cramped Hilbert modules over $\cA(\theta(\bD^d))$ by $\mathfrak{C}$.
\end{definition}
Projective modules are $cornerstones$ (to borrow from \cite{D-P}) for studying general modules in homological algebra. A well-known characterization which stems from general module theory over rings and which we can take as the definition of projectivity is as follows.
\begin{definition}\label{L:Projective}
A Hilbert module $\cP$ in the category $\mathfrak{H}$ or in the category $\mathfrak{C}$ is projective  if and only if every short exact sequence of the form 
$$ 0 \longrightarrow \cH \longrightarrow \cK \longrightarrow \cP \longrightarrow 0,$$
where $\cH$ and $\cK$ are in the same category, is a split exact sequence. 
\end{definition}
The definition above leads to a measure of non-projectivity. Let $\mathcal{H}$ and $\mathcal{K}$ be objects in category $\mathfrak{H}$ of Hilbert modules over $\cA(\theta(\bD^d))$. Let $\mathcal{S}\left(\mathcal{K},\mathcal{H}\right)$ be the set of all short exact sequences of the form $$ E: 0\longrightarrow \mathcal{H} \xlongrightarrow{\alpha} \cJ \xlongrightarrow{\beta} \mathcal{K} \longrightarrow 0,$$ where $\cJ$ is an object in $\mathfrak{H}$. Let $\cJ'$ be an object in $\mathfrak{H}$ and 
$$E':0\longrightarrow \mathcal{H} \xlongrightarrow{\alpha'} \cJ' \xlongrightarrow{\beta'} \mathcal{K} \longrightarrow 0 $$ be another short exact sequence. Here $\alpha, \beta, \alpha'$ and $\beta'$ are Hilbert module maps. We say $E$ is \textit{equivalent} to $E'$ if there is a Hilbert module map $\gamma: \cJ\to \cJ'$ such that the diagram 

\[
  \begin{tikzcd}
    E: 0 \arrow{r} & \mathcal{H} \arrow{r}{\alpha}  \arrow[d, equal] & \cJ \arrow{r}{\beta} \arrow[d, dashed, "\gamma"] & \mathcal{K} \arrow{r} \arrow[d, equal] & 0 \\
    E': 0 \arrow{r} & \mathcal{H} \arrow{r}{\alpha'} & \cJ' \arrow{r}{\beta'} & K \arrow{r} & 0
  \end{tikzcd}
\] commutes. This becomes an equivalence relation. It can be deduced using the exactness and the commutativity of the diagram that $\gamma$ is actually an invertible map. The set of equivalence classes of $\mathcal{S}(\mathcal{K}, \mathcal{H})$ under this relation is defined to be the cohomology group $\operatorname{Ext_{\mathfrak{H}}(\mathcal{K}, \mathcal{H})}$, called as \textit{extension group}. It is a group under an addition rule, known as the \textit{Baer sum}; see \cite[Chapter III]{Maclane}. The ``zero" of this additive group is the short exact sequence 
$$ 0 \longrightarrow \mathcal{H} \xlongrightarrow{i} \mathcal{H}\oplus \mathcal{K} \xlongrightarrow{\pi} \mathcal{K} \longrightarrow 0$$ 
where $\mathcal{H}\oplus \mathcal{K}$ (Hilbert space direct sum) is the Hilbert module over $\cA(\theta(\bD^d))$ under the action $g \cdot (h,k)= (g\cdot h,  g \cdot k)$, $i(h):= (h,0)$ is the inclusion, and $\pi(h,k)= k$ is the projection on $\mathcal{K}$ for $h\in \mathcal{H}, k\in \mathcal{K}$. In a similar way we can define extension group in the cramped category $\mathfrak{C}$ for two cramped Hilbert modules $\mathcal{H}$ and $\mathcal{K}$; it is denoted by $\operatorname{Ext}_{\mathfrak{C}}(\mathcal{K}, \mathcal{H})$. This leads to a re-statement of the definition of projectivity: A Hilbert module $\cP$ is projective in $\mathfrak{H}$ or $\mathfrak{C}$ if and only if $\textup{Ext}_{\mathfrak{H}}(\cP, \cH)=\{0\}$ for any Hilbert module $\cH$ in the same category.

\begin{itemize}
\item We now briefly describe the main results of Part A. Every bounded symmetric domain $\Omega$ has a Hardy space associated with it. Corresponding to any one-dimensional representation $\rho$ of a pseudo-reflection group $G \subset $ Aut$(\Omega)$, a Hardy space $H^2_\rho(\theta(\Omega))$ of functions on $\theta(\Omega)$ is known. Using its identification as a subspace of $L^2(\partial \theta(\Omega), \mu_{\rho, \theta})$, where $\mu_{\rho, \theta}$ is the push-forward of a natural measure from the \v Silov boundary of the bounded symmetric domain $\Omega$, we define small Hankel operators in  \cref{smallH} for any $H^2_\rho(\theta(\Omega))$ symbol. The key concept of the weak product gives us the two main results of \cref{smallH}: 
    \begin{theorem} \leavevmode
    \begin{enumerate}
    \item The space of symbols which define bounded small Hankel operators is isometrically isomorphic to the dual of the weak product of $H^2_\rho(\theta(\Omega))$ with itself. 
    
    \item A small Hankel operator is compact if and only if the symbol is in the closure of holomorphic polynomials in the above mentioned dual. 
    \end{enumerate}
    \end{theorem}
    
    The big Hankel operator is the object of study in \cref{Big-Hankel} and the dominant theme is the failure of Nehari's Theorem. Two main results are:
      \begin{theorem} \leavevmode
    \begin{enumerate}
\item Under a regularity condition on the Szeg\"o kernel of the quotient domain, if a function $f \in L^2(\partial \theta(\Omega), \mu_{\rho, \theta})$ defines bounded big Hankel operators $H_f$ and $H_{\bar{f}}$ , then $f$ is in BMO of the \v Silov boundary of the quotient domain.
\item Nehari's theorem does not hold for big Hankel operators in general on quotient domains.
   \end{enumerate}
\end{theorem}
Expectedly, all these have algebraic applications which are the contents of the next part. 
    
    \item  The broad aim of Part B is to find projective Hilbert modules. Whether there exist non-zero projective Hilbert modules over any function algebra was asked in the seminal work \cite{D-P}. This was first answered in \cite{Ca-Cl-Fo-Wi} where projective Hilbert modules over $\mathcal A(\bD)$ were found. The only other known examples of projective Hilbert modules are over the polydisc algebra $\mathcal A(\bD^d)$, see \cite{Car-Cla2}. Hence, there may be a need to restrict the category. Then there is success in \cite{Cl-Mc}. We shall impose a topological condition as in \cite{Guo-Studia} to consider the so-called normal category. 
         
         \cref{Pisier} deals with the two categories $\mathfrak{H}$ and $\mathfrak{C}$ and shows that they are not equal by means of an example \`a la Pisier. 
         
         The main result of \cref{Application} says that
         \begin{theorem}
            The Hardy space is not a projective object in the category $\mathfrak{H}$ or in the category $\mathfrak{C}$.
            \end{theorem}
            The proof of this theorem requires material from Part A, viz., the failure of Nehari's theorem.  In fact, no projective objects are known in this context. 
         
         On the other hand, \cref{normal_cat} shows that
          \begin{theorem}
         Every normal Hilbert module over $C(\partial \theta(\bD^d))$ is a projective object in the category of normal Hilbert modules over $\cA(\theta(\bD^d))$.
          \end{theorem} These terminologies will be defined at appropriate places. To our pleasant surprise, classical function theory - for example Carath\'eodory approximation - plays a role in the study of normal modules.
\end{itemize}

\part{Boundedness and compactness of Hankel operators}
\section{Hardy spaces of quotients of bounded symmetric domains}\label{Hardy}
 {\em All pseudo-reflection groups $G$ in this note are finite.} 

This section delineates the features of the Hardy space for our purpose. A good portion of this section has appeared in the recent preprint \cite{GSR}. When $\Omega$ is a bounded symmetric domain, as it will be from now on, Hardy spaces are well-studied \cite{Hahn-Mitchell, Faraut-Koranyi, Upmeier}. Let $\nu$ be the unique normalized isotropy invariant measure on the \v Silov boundary $\partial \Omega$. The Hardy space is defined as 
\begin{align*}
H^2(\Omega)= \left\lbrace f: f \text{ is holomorphic in } \Omega \text{ and } \sup_{0\leq r <1} \int_{\partial \Omega} |f(r \boldsymbol{\zeta})|^2 d\nu(\boldsymbol{\zeta}) < \infty  \right\rbrace.
\end{align*}
Theorem 6 in \cite{Hahn-Mitchell} notes that  $H^2(\Omega)$  can be isometrically embedded  into $L^2(\partial \Omega, \nu)$. Moreover, $H^2(\Omega)$ has an orthonormal basis which we shall denote by 
$$\{e_\alpha^{(k)}: 1 \leq \alpha \leq N_k=(\begin{smallmatrix} d+k-1 \\ k \end{smallmatrix}) \text{ and } k \in \mathbb{N}\}.$$ 
Here for each $k$, $\{e_\alpha^{(k)}\}$ is the collection of $N_k$ homogeneous polynomials which are unique linear combinations of monomials of degree $k$ in $\bC[\boldsymbol{z}]$. We are not writing the explicit polynomial because we do not need it. The inquisitive reader may refer to page 78 of \cite{Hua}.  Consequently, by Zaremba formula, $H^2(\Omega)$ is a reproducing kernel Hilbert space. The kernel will be denoted by $S_\Omega(\boldsymbol{z}, \boldsymbol{w})$. It is known, see for example page 25 in \cite{Ding} that $S_\Omega(\boldsymbol{z}, \boldsymbol{w})$ and $(S_\Omega(\boldsymbol{z}, \boldsymbol{w}))^{-1}$ are in $\mathcal A (\Omega)$ for every $\boldsymbol{w}$ in $\Omega$. 

When $\Omega$ is a $G$-invariant domain for a pseudo-reflection group $G$, consider the collection $\hat{G}_1$ of one dimensional representations of $G$. For every $\rho$  in the collection $\widehat{G}_1$ of one-dimensional representations of $G$, consider 
$$ R_\rho^G\left( \bC[\boldsymbol{z}] \right)=\left\lbrace f\in \bC[\boldsymbol{z}]: \sigma(f)=\rho(\sigma) f \text{ for all } \sigma \in G \right \rbrace.$$
It is well-known that there exists a homogeneous polynomial $\ell_\rho$ such that each $f \in R_\rho^G\left( \bC[\boldsymbol{z}] \right) $ can be expressed as $f= \ell_\rho \hat{f}$ for some $G$-invariant polynomial $\hat{f}$, see \cite[Theorem 3.1] {Stanley}.   Let $d\mu_{\rho, \theta}$ be the measure on the \v Silov boundary $\partial\theta(\Omega)$ of $\theta(\Omega)$ defined as
$$ \int_{\partial \theta( \Omega)} f  d\mu_{\rho, \theta}=  \frac{1}{|G|}\int_{\partial \Omega} f \circ \theta |\ell_\rho|^2 d\nu.$$
\begin{definition}
Corresponding to the one-dimensional representation $\rho$, the Hardy space of the quotient domain is defined as
$$
H_\rho ^2( \theta(\Omega))=\left \lbrace f:\theta(\Omega) \to \bC \text{ holomorphic and } \ell_\rho (f \circ \theta) \in H^2(\Omega)
  \right \rbrace.
$$
\end{definition}
The space $H_\rho ^2( \theta(\Omega))$ is a reproducing kernel Hilbert space (the kernel is given by \eqref{Reproducing kernel quotient}). As a reproducing kernel Hilbert space, it has an identification with a certain subspace 
 $$R_\rho^G \left( H^2(\Omega)\right) = \left \lbrace f\in H^2(\Omega): \sigma(f)= \rho(\sigma) f, \text{ for all } \sigma \in G \right \rbrace $$ 
 of $ H^2(\Omega)$ via the isomorphism 
 \begin{align}\label{Gamma_rho}
\Gamma_\rho(f) = \frac{1}{\sqrt{|G|}} \ell_\rho (f\circ \theta) \text{ for } f \in H_\rho ^2( \theta(\Omega)). 
\end{align}
Indeed, if $S_\Omega$ is the reproducing kernel for $H^2(\Omega, \mu)$, define the $G$-invariant function 
\begin{align*}
(\boldsymbol{z}, \boldsymbol{w}) \mapsto \frac{1}{\ell_\rho(z) \overline{\ell_\rho(w)}} \sum_{\sigma \in G} \overline{\rho(\sigma)} S_\Omega (\sigma(z), w)
\end{align*} 
on $\Omega \times \Omega$ which is holomorphic in the first argument and anti-holomorphic in the second. By the analytic Chevalley-Shephard-Todd theorem, there is a function $S_{\theta,\rho} $ defined on $\theta(\Omega) \times \theta(\Omega)$ which agrees with the function above.  The one-dimensional representation $\rho$ induces an orthogonal projection $\mathbb{P}_\rho : L^2(\partial \Omega, \nu) \rightarrow L^2(\partial \Omega, \nu)$ defined as
\begin{align}\label{P-rho}
\mathbb{P}_\rho (\psi) = \frac{1}{|G|} \sum_{\sigma \in G} \rho(\sigma^{-1}) \sigma(\psi), \text{ for } \psi \in L^2(\partial \Omega, \nu)
\end{align} 
and a simple calculation shows that
\begin{align}\label{Reproducing kernel quotient}
S_{\rho,\theta} (\theta(\boldsymbol{z}), \theta(\boldsymbol{w})) = \frac{|G|}{\ell_\rho(z) \overline{\ell_\rho(w)}} \mathbb{P}_\rho S_\Omega(\boldsymbol{z}, \boldsymbol{w}).
\end{align}
Then $\Gamma_\rho (S_{\rho,\theta}(\cdot, \theta(\boldsymbol{w})))(\boldsymbol{z}) = \frac{1}{\sqrt{|G|}} \ell_\rho(\boldsymbol{z})S_{\rho,\theta}(\theta(\boldsymbol{z}), \theta(\boldsymbol{w}))$ and for any $f \in H^2_\rho(\theta(\Omega))$ and $\boldsymbol{w} \in \Omega$ we have
\begin{align*}
\left \langle f, S_{\rho,\theta}(\cdot, \theta(\boldsymbol{w})) \right \rangle &= \frac{1}{|G|} \left \langle \ell_\rho (f\circ \theta), \ell_\rho S_{\rho,\theta}(\theta(\cdot), \theta(\boldsymbol{w})) \right \rangle = \frac{1}{|G|} \left \langle \ell_\rho (f\circ \theta), \frac{|G|}{\overline{\ell_\rho(\boldsymbol{w})}} \mathbb{P}_\rho S_\Omega(\cdot, \boldsymbol{w})\right \rangle \\
&= \frac{1}{\ell_\rho(\boldsymbol{w})} \left \langle \ell_\rho (f\circ \theta), S_\Omega(\cdot, \boldsymbol{w}) \right \rangle = f\circ \theta(\boldsymbol{w}). 
\end{align*}
This shows that $S_{\rho,\theta}$ is the reproducing kernel of $H^2_\rho(\theta(\Omega))$.
When we restrict  $\mathbb{P}_\rho$ to $ H^2(\Omega)$, the range  is the space $R_\rho^G \left( H^2(\Omega)\right)$ defined above. We emphasize these identifications of the Hardy space of the quotient domain corresponding to the one-dimensional representation $\rho$. 

Since $\theta$ preserves the \v Silov boundary, an isometric map $\Gamma^\prime_\rho$ analogous to \eqref{Gamma_rho} can be defined on $ L^2 (\partial\theta(\Omega), \mu_{\rho, \theta})$ with the range
$$ R_\rho^G \left( L^2(\partial \Omega, \nu)\right)= \left \lbrace f\in L^2(\partial \Omega, \nu): \sigma(f)= \rho(\sigma) f, \text{ for all } \sigma \in G \right \rbrace .$$ 
This shows that the Hardy space of the quotient domain corresponding to the one-dimensional representation $\rho$ can be isometrically embedded  as a subspace of the $L^2$-space
$$ 
L^2 (\partial\theta(\Omega), \mu_{\rho, \theta})= \left \lbrace f: \partial \theta(\Omega) \to \bC \text{ is } \mu_{\rho, \theta} \text{ measurable and } \int_{\partial \theta(\Omega)} |f|^2 d\mu_{\rho, \theta} <\infty \right \rbrace .
$$
Indeed, we use the following diagram: 
\[
  \begin{tikzcd}
     H_\rho ^2( \theta(\Omega))  \arrow{d}{\Gamma_\rho} \arrow{r}{\Gamma^{\prime*}_\rho \circ W \circ \Gamma_\rho} & L^2(\partial \theta(\Omega), \mu_{\rho, \theta}) \arrow{d}{\Gamma^\prime_\rho} \\
     R_\rho^G \left( H^2(\Omega)\right)  \arrow{r}{W} & R_\rho^G \left( L^2(\partial\Omega, \nu) \right) 
  \end{tikzcd}
\]
where $W$ is the isometry which sends a function in $R_\rho^G \left( H^2(\Omega)\right)$ to its radial limit (or the boundary value) in $R_\rho^G \left( L^2(\partial\Omega, \nu) \right)$, i.e., $W(f)(\bm\zeta)=\lim_{r\to1-}f(r\bm\zeta)$ for almost all $\bm\zeta\in\partial\Omega$ with respect to $\nu$. Such limit is known to exist almost everywhere with respect to $\nu$; see \cite[Theorem 6]{Hahn-Mitchell}. 

\begin{lemma} \label{Radial}
The isometry $\Gamma^{\prime*}_\rho \circ W \circ \Gamma_\rho$ takes a $H_\rho ^2( \theta(\Omega))$ member $f$ to its radial limit $f^*$ given by
$$
f^*(\bm q)=\lim_{r\to1-}f\circ\theta(r\bm\zeta) \quad\mu_{\rho,\theta}\; a.e.\ \mbox{ where }\theta(\bm\zeta)=\bm q.
$$ Furthermore, it takes the $H^\infty$-subspace of $H^2_\rho(\theta(\Omega))$
$$
H^\infty(\theta(\Omega))=\{f:\theta(\Omega)
\to\bC: f \mbox{ is analytic and bounded}\}
$$into the $L^\infty$-subspace of $L^2(\partial \theta(\Omega), \mu_{\rho, \theta})$
$$
L^\infty(\partial \theta(\Omega), \mu_{\rho, \theta})=\{f:\partial \theta(\Omega)\to\bC: f \mbox{ is measurable and essentially bounded w.r.t.\ }\mu_{\rho,\theta}\}.
$$
\end{lemma}

\begin{proof} We need to show the existence of the radial limit. Given $f \in H^2_\rho(\theta(\Omega))$, $\ell_\rho (f\circ \theta) \in R^G_\rho(H^2(\Omega))$, the radial limit 
\begin{align}\label{Eq:radial limit}
(\ell_\rho (f\circ \theta))^*(\boldsymbol{\zeta}) = \lim_{r\to1-} \ell_\rho (r\boldsymbol{\zeta}) (f\circ \theta)(r\boldsymbol{\zeta})= \ell_\rho(\boldsymbol{\zeta}) \lim_{r\to1-} (f\circ \theta)(r\boldsymbol{\zeta}) 
\end{align}
exists for $\nu$-a.e. $\boldsymbol{\zeta} \in \partial \Omega$ and $(\ell_\rho (f\circ \theta))^* \in R^G_\rho L^2(\partial \Omega, \nu)$. In particular $(\ell_\rho (f\circ \theta))^*$ is $\nu$-measurable and so, $|\ell_\rho|^2 d\nu$-measurable. If $V$ is the set of points in $\bT^d$ where the above radial limit does not exist, then by $G$-invariance of $f\circ \theta$, it is clear that $V$ is a $G$-invariant subset of $\bT^d$ i.e., $\sigma(V)= V$ for every $\sigma \in G$. Therefore $\theta^{-1}(\theta(V))=V$. Again,
$$
\mu_{\rho, \theta}(\theta(V))= |\ell_{\rho}^2| d\nu (\theta^{-1}(\theta(V))) = \int_{V} |\ell_{\rho}^2| d\nu =0
$$ 
and so, $\theta(V)$ is $\mu_{\rho, \theta}$-measurable. Again, $\nu(\mathcal{Z}(\ell_\rho) \cap \partial \Omega)=0$ and hence $(|\ell_\rho|^2 d\nu)(\mathcal{Z}(\ell_\rho) \cap \partial \Omega)=0$ where $\mathcal{Z}(\ell_\rho)$ is the zero set of the polynomial $\ell_\rho$ \cite{Stoll}. Now, for $\boldsymbol{q}\in \partial \theta(\Omega)\setminus (\theta(V) \cup \theta(\mathcal{Z}(\ell_\rho))) $ we define our candidate for the radial limit in the quotient domain by 
\begin{align}\label{Eq:radial limit quotient}
f^*(\boldsymbol{q}) = \frac{(\ell_\rho (f \circ \theta))^*(\boldsymbol{\zeta})}{\ell_\rho(\boldsymbol{\zeta})} \ \text{ where } \theta(\boldsymbol{\zeta})= \boldsymbol{q}.
\end{align}
Note that $f^*$ is well-defined as the set $V$ and $\mathcal{Z}(\ell_\rho)$ are $G$-invariant. We define $f^*$ to be zero on $\theta(V) \cup \theta(\mathcal{Z}(\ell_\rho) \cap \partial \Omega)$. To show that $f^*$ is $\mu_{\rho, \theta}$-measurable, take an open set $U$ in $\mathbb{C}$. Since $d\mu_{\rho, \theta}$ on $\partial \theta(\bD^d)$ is the push-forward of the measure $|\ell_\rho^2| d\nu$ on $\bT^d$ under the map $\theta$, ${f^*}^{-1}(U)$ is $\mu_{\rho, \theta}$-measurable if and only if $\theta^{-1}({f^*}^{-1}(U))= (f^* \circ \theta)^{-1}(U)= (\frac{(\ell_\rho (f\circ \theta))^*}{\ell_\rho})^{-1}(U)$ is $|\ell_\rho^2| d\nu$-measurable. Since $\nu(\mathcal{Z}(\ell_\rho) \cap \partial \Omega)=0$, the function $\frac{1}{\ell_\rho}$ is $\nu$-measurable and $|\ell_\rho|^2 d\nu$-measurable. Therefore, $\frac{(\ell_\rho (f\circ \theta))^*}{\ell_\rho}$ is $\nu$-measurable as well as $|\ell_\rho|^2 d\nu$-measurable. Thus $f^*$ is $\mu_{\rho, \theta}$-measurable and $f^* \in L^2(\partial \Omega, \mu_{\rho, \theta})$ with $\|f\|= \|f^*\|$.

Similarly, for $f \in H^\infty(\theta(\Omega)) \subset H^2_\rho(\theta(\Omega))$ we have $f^* \in L^2(\partial \theta(\Omega), \mu_{\rho, \theta})$. Moreover, $(f\circ \theta)^*(\boldsymbol{\zeta})= \lim_{r\to1-} (f\circ \theta)(r\boldsymbol{\zeta})$ exists $\nu$-a.e. $\boldsymbol{\zeta}$ on $\partial \Omega$ as $f \circ \theta$ is in $H^\infty(\Omega)$. Thus from \eqref{Eq:radial limit} and \eqref{Eq:radial limit quotient} we get
$$
f^*(\boldsymbol{q}) = (f\circ \theta)^*(\boldsymbol{\zeta}) \ \text{ where } \theta(\boldsymbol{\zeta}) = \boldsymbol{q} 
$$
for $\mu_{\rho, \theta}$-a.e. $\boldsymbol{q}$ in $\partial \theta(\Omega)$. 

Moreover, $(f\circ \theta)^* \in L^\infty(\bT^d, \nu)$ with $\|f\|_{H^\infty(\theta(\Omega))}=\|f\circ \theta\|_{H^\infty(\Omega)} = \|(f\circ \theta)^*\|_{L^\infty(\partial\Omega, \nu)}$. The measure $|\ell_\rho^2| d\nu$ being absolutely continuous with respect to $\nu$, $(f\circ \theta)^* \in L^\infty(\bT^d, |\ell_\rho^2| d\nu)$ and so, $f^*\in L^\infty(\partial \theta(\Omega), \mu_{\rho, \theta})$ with $M = \|f^*\|_{L^\infty(\partial\theta(\bD^d), \mu_{\rho, \theta})} \leq \|(f\circ \theta)^*\|_{L^\infty(\partial\Omega, \nu)}$. If $M< \|(f\circ \theta)^*\|_{L^\infty(\partial\Omega, \nu)}$, then for set
$$ E= \{ \boldsymbol{\zeta}\in \partial \Omega: M< |(f\circ \theta)^*(\boldsymbol{\zeta})| \leq \|(f\circ \theta)^*\|_{L^\infty(\partial\Omega, \nu)} \}, $$
$\nu(E)>0$. Also, $\nu(E \setminus \mathcal{Z}(\ell_\rho))>0$. By inner regularity of $\nu$, there exists a compact subset $C$ of $\partial \Omega$ such that $C \subset E \setminus \mathcal{Z}(\ell_\rho)$ with $\nu(C)>0$. Therefore,
$$
|\ell_\rho|^2 d\nu (E) \geq \int_{C} |\ell_\rho|^2 d\nu >0 
$$
as $\inf_C |\ell_\rho|^2 >0$. Note that $E$ is $G$-invariant set and hence $\theta^{-1}(\theta(E))=E$. This implies that $\mu_{\rho, \theta}(\theta(E))= |\ell_\rho|^2 d\nu (E) >0$. This is absurd as $\theta(E)$ consists of all points $\boldsymbol{q}$ in $\partial \theta(\Omega)$ such that $\|f^*\|_{L^\infty(\partial \theta(\Omega), \mu_{\rho, \theta})} < \|f^*(\boldsymbol{q})\|$. Hence we must have $\|f^*\|_{L^\infty(\partial \theta(\Omega), \mu_{\rho, \theta})} = \|(f\circ \theta)^*\|_{L^\infty(\partial \Omega, \nu)}$. This implies that $f \mapsto f^*$ is an isometric embedding of $H^\infty(\theta(\bD^d))$ into $L^\infty(\partial \theta(\bD^d), \mu_{\rho, \theta})$.
  
\end{proof}

Let $M_{p_j}$ be the operator of co-ordinate multiplication by $p_j$ on $L^2(\partial\theta(\Omega))$ for $1\leq j \leq d$. Let $T_{p_j}=  M_{p_j}|_{H_\rho ^2( \theta(\Omega))} $ be the $j$-th co-ordinate multiplier on $H_\rho ^2( \theta(\Omega))$. It is clear that the operators $M_{p_j}$ and $T_{p_j}$ are unitary equivalent to $M_{\theta_j}$ on $R_\rho^G L^2(\partial \Omega)$ and $T_{\theta_j}$ on $R_\rho^G(H^2(\Omega))$ respectively via the unitary map $\Gamma^\prime_\rho$ and $\Gamma_\rho$ respectively.

\section{The small Hankel operator} \label{smallH}
\begin{definition} 
For $\varphi \in H_\rho^2(\theta(\Omega))$, the \textit{small Hankel operator} $h_\varphi$ on $H_\rho^2(\theta(\Omega))$ is a conjugate-linear densely defined operator defined as $$ h_\varphi (f) \defeq P_{H_\rho^2(\theta(\Omega))}(\varphi \bar{f}),\text{ for all holomorphic polynomials } f\in H_\rho^2(\theta(\Omega)),$$ where $P_{H_\rho^2(\theta(\Omega))}$ is the orthogonal (Szeg\"o) projection from $L^2(\partial \theta(\Omega),\mu_{\rho, \theta})$ onto $H_\rho^2(\theta(\Omega))$.
\end{definition}
Although our small Hankel operator is conjugate-linear, some authors prefer to view the small Hankel operator as a linear operator. They take the complex-conjugate of $h_\varphi (f)$ and hence the range of the operator in that case is the complex-conjugate of $H_\rho^2(\theta(\Omega))$.

The results of this section, viz., characterization of boundedness and compactness of small Hankel operators via the dual space of a weak-product space are motivated by \cite{CRW} and \cite{Krantz-Li}.

 Characterization of the boundedness of small Hankel operators is related to {\em factorization of $H^1$ functions}. In the case of the open unit disc $\bD$, the Riesz factorization of $H^1(\bD)$ functions plays a crucial role in proving Nehari's theorem. Exact analogues of Nehari's theorem have been proved in the case of the Euclidean unit ball \cite{CRW} and more generally on smoothly bounded strongly pseudoconvex domains \cite{Krantz-Li}. The characterizations in these cases have been obtained via a weak factorization and atomic decompositions. An analogue in the case of bidisc is still unknown.  

\begin{thm}\label{alg-characterization}
Suppose $A: H_\rho^2(\theta(\Omega)) \to H_\rho^2(\theta(\Omega)) $ be a bounded conjugate-linear operator. Then, $A$ is a small Hankel operator if and only if $T_{p_j}^* A = A T_{p_j}$ for $j=1,\dots, d$.
\end{thm}
\begin{proof}
Suppose $A = h_\varphi$ is a bounded small Hankel operator on $H_\rho^2(\theta(\Omega))$ with $\varphi = A(1)$. For $f,g \in H_\rho^2(\theta(\Omega))$, 
\begin{align*}
\langle T_{p_j}^* h_\varphi(f), g \rangle = \langle P_{H_\rho^2(\theta(\Omega))}(\varphi \bar{f}), p_j g \rangle = \int_{\partial \theta(\Omega)} \varphi \overline{f p_j g} \ d\mu_{\rho, \theta} \text{ and }
\end{align*}
\begin{align*}
\langle h_\varphi T_{p_j} f, g \rangle = \langle P_{H_\rho^2(\theta(\Omega))}(\varphi \bar{p_j} \bar{f}), g \rangle = \int_{\partial \theta(\Omega)} \varphi \overline{f p_j g}\  d\mu_{\rho, \theta}.
\end{align*} 
Therefore $T_{p_j}^* h_\varphi = h_\varphi T_{p_j}$. 

Conversely, assume that $A$ is bounded conjugate-linear operator satisfying the given algebraic relations. It is easy to see that $T_\psi^* A = A T_\psi$ for every polynomial $\psi$ in $p_1,\dots, p_d$. Define, $\varphi = A(1)$. Let $f$ and $g$ be any two polynomials in $p_1,\dots, p_d$. Then 
\begin{align*}
\langle A(f), g \rangle = \langle A T_f (1), g \rangle = \langle T_f^* A(1), g \rangle = \langle \varphi, fg \rangle= \langle h_\varphi (f), g \rangle.
\end{align*} 
Density of the polynomials in $p_1,\dots, p_d$ in the space $H_\rho^2(\theta(\Omega))$ implies that $A= h_\varphi$.
\end{proof}

The {\em weak product} space of $H_\rho^2(\theta(\Omega))$ is defined by 
$$ 
H_\rho^2(\theta(\Omega)) \odot H_\rho^2(\theta(\Omega)) = \left \lbrace  \sum_{j=1}^\infty f_j g_j : \sum_{j=1}^\infty \|f_j\| \|g_j\| <\infty \text{ where } f_j, g_j \in H_\rho^2(\theta(\Omega))  \right \rbrace
$$
where the norm is given by 
$$
\|\psi\|_\odot = \inf \left \lbrace \sum_{j=1}^\infty \|f_j\| \|g_j\|: \psi=\sum_{j=1}^\infty f_j g_j  \right \rbrace.
$$
It can be checked that this is a Banach space of analytic functions on $\theta(\Omega)$ and the point evaluations at points of $\theta(\Omega)$ are continuous. The idea stems from \cite{CRW}. See \cite{Hartz, Martin} for a recent works on weak product spaces in more general settings.
Following \cite{Richter-Sundberg}, we also consider the holomorphic polynomials $\bC[\boldsymbol{z}]$ which are dense in $H^2(\Omega)$. Define 
\begin{align*}
\bC[\boldsymbol{p}] \bullet \bC[\boldsymbol{p}]= \left\lbrace \sum_{j=1}^n g_j q_j : n\in \mathbb{N}, g_j, q_j \in \bC[\boldsymbol{p}] \right \rbrace.
\end{align*}
For $f \in \bC[\boldsymbol{p}] \bullet \bC[\boldsymbol{p}]$, we define 
\begin{align*}
\|f\|_\bullet = \inf \left \lbrace \sum_{j=1}^n \|g_j \| \|q_j\|: f= \sum_{j=1}^n g_j q_j \text{ with } g_j , q_j \in \bC[\boldsymbol{p}]  \right \rbrace.
\end{align*}
Then $(\bC[\boldsymbol{p}] \bullet \bC[\boldsymbol{p}], \|\cdot\|_\bullet)$ is a normed linear space and also $\|f\|_\odot \leq \|f\|_\bullet $ for $f$ in $\bC[\boldsymbol{p}] \bullet \bC[\boldsymbol{p}]$. Therefore, the inclusion $j: \bC[\boldsymbol{p}] \bullet \bC[\boldsymbol{p}] \rightarrow  H_\rho^2(\theta(\Omega)) \odot H_\rho^2(\theta(\Omega))$ extends as a contraction from the completion $(\bC[\boldsymbol{p}] \bullet \bC[\boldsymbol{p}])_\bullet $ of $\bC[\boldsymbol{p}] \bullet \bC[\boldsymbol{p}]$ with respect to $\|\cdot \|_\bullet $ to $H_\rho^2(\theta(\Omega)) \odot H_\rho^2(\theta(\Omega))$. We denote this extension by $J$.
\begin{proposition} \label{P: Norm equality}
For $f \in \bC[\boldsymbol{p}] \bullet \bC[\boldsymbol{p}]$, $\|f\|_\bullet = \|f\|_\odot$.
\end{proposition}
\begin{proof} 
Since $\bC[\boldsymbol{p}]$ is dense in $H_\rho^2(\theta(\Omega))$, the map $\tilde{J}: (\bC[\boldsymbol{p}] \bullet \bC[\boldsymbol{p}])_\bullet / \operatorname{Ker}(J) \rightarrow H_\rho^2(\theta(\Omega)) \odot H_\rho^2(\theta(\Omega)) $, induced by $J$ is onto and isometric \cite[Theorem 1.1]{Richter-Sundberg}. This shows that for $f \in \bC[\boldsymbol{p}] \bullet \bC[\boldsymbol{p}]$, $\|f\|_\bullet = \|f\|_\odot$ if and only if $\operatorname{Ker}J$ is trivial. By an argument motivated from the proof of Theorem 1.2 of \cite{Richter-Sundberg}, we shall show that $\operatorname{Ker}(J)=\{0\}$.

Let $m_j$ denote the degree of the homogeneous polynomial $\theta_j$ for $1\leq j \leq d$ and $m_0$ is the degree of the homogeneous polynomial $\ell_\rho$. Denote $\boldsymbol{m} =(m_1, \cdots, m_d)$ and $r^{\boldsymbol{m}} \boldsymbol{p}= (r^{m_1}p_1,\dots, r^{m_d} p_d)$.
Since a bounded symmetric domain is star-like about origin and since the $\theta_j$ are homogeneous polynomials,  define 
$$ f_r( \boldsymbol{p}) \defeq r^{2m_0}f(r^{\boldsymbol{m}} \boldsymbol{p}) \ \text{ for all } f\in H_\rho^2(\theta(\Omega)),  \boldsymbol{p} \in \overline{\theta(\Omega)} \text{ and }  r\in (0,1).$$

\noindent Therefore, 
$$\Gamma_\rho (f_r)(\boldsymbol{z}) =\frac{1}{\sqrt{|G|}} r^{2m_0} \ell_\rho (\boldsymbol{z}) f(r^{m_1}\theta_1(\boldsymbol{z}),\dots, r^{m_d} \theta_d(\boldsymbol{z}))= \frac{1}{\sqrt{|G|}} r^{m_0} (\ell_\rho f\circ \theta)(r \boldsymbol{z})$$ 
for $\boldsymbol{z} \in \Omega$.
Clearly, $f_r$ is holomorphic in a neighbourhood of $\overline{\theta(\Omega)}$ and hence the corresponding boundary value functions are in $C(\partial \theta(\Omega))$. Consider the linear operator 
$$\delta_r : H_\rho^2(\theta(\Omega)) \to H_\rho^2(\theta(\Omega)) \quad \text{ such that } \quad \delta_r (f) = f_r.$$ 
We list some properties of these operators $\{\delta_r \}$ as a lemma below and resume the proof of the proposition after it.

\begin{lemma}\label{Hart-L2} \leavevmode
\begin{itemize}
\item[(i)] $\| \delta_r\| \leq 1$ for all $r\in [0,1)$.
\item[(ii)] Each $\delta_r $ is self-adjoint. 
\item[(iii)] As $r \to 1^-$, $\delta_r$ converges to the identity operator on ${H_\rho^2(\theta(\Omega))}$ in strong operator topology. 
\item[(iv)] $\delta_r$ extends as a contraction with respect to $\|\cdot\|_\bullet$ from $\bC[\boldsymbol{p}] \bullet \bC[\boldsymbol{p}]$ to its completion. Moreover, $\delta_r(f) \xrightarrow{\|\cdot\|_\bullet} f$ for $f \in (\bC[\boldsymbol{p}] \bullet \bC[\boldsymbol{p}])_\bullet$.
\end{itemize}
\end{lemma}
\begin{proof}
The first one is quite obvious. For the second one, take $f, g\in H_\rho^2(\theta(\Omega))$. Then
\begin{align*}
\langle \delta_r(f), g \rangle & = \langle \Gamma_\rho(f_r), \Gamma_\rho(g) \rangle \\ 
                               &= \frac{1}{|G|}\int_{\partial \Omega} r^{m_0} (\ell_\rho f\circ \theta)(r\boldsymbol{\zeta})\overline{(\ell_\rho g\circ \theta)} (\boldsymbol{\zeta}) d\nu(\boldsymbol{\zeta})\\
                               &= \frac{r^{m_0}}{|G|} \sum_{k=0}^\infty \sum_{\alpha=1}^{N_k} r^k a_\alpha^{(k)} \bar{b}_\alpha^{(k)} \\
                               &=\frac{1}{|G|} \int_{\partial \Omega} r^{m_0} (\ell_\rho f\circ \theta)(\boldsymbol{\zeta})\overline{(\ell_\rho g\circ \theta)} (r\boldsymbol{\zeta}) d\nu(\boldsymbol{\zeta}) \\
                               & = \langle \Gamma_\rho(f), \Gamma_\rho(g_r) \rangle = \langle f, \delta_r(g) \rangle
\end{align*}
where the fourth step is obtained by considering the expressions (in terms of the orthonormal basis $\{e_\alpha^{(k)}\}$ consisting of homogeneous polynomials) $\ell_\rho (f\circ \theta) = \sum_{k=0}^\infty \sum_{\alpha=1}^{N_k} a_\alpha^{(k)} e_\alpha^{(k)} $ and $\ell_\rho (g\circ \theta) =\sum_{k=0}^\infty \sum_{\alpha=1}^{N_k} b_\alpha^{(k)} e_\alpha^{(k)} $ with $\sum_{k=0}^\infty \sum_{\alpha=1}^{N_k} |a_\alpha^{(k)}|^2 = \|\ell_\rho (f\circ \theta)\|_{H^2(\Omega)} $ and $\sum_{k=0}^\infty \sum_{\alpha=1}^{N_k} |b_\alpha^{(k)}|^2 = \|\ell_\rho (g\circ \theta)\|_{H^2(\Omega)}$.

To prove (iii), for $f\in H_\rho^2(\theta(\Omega))$ we have,
$$ \|\delta_r(f) - f\|^2  = \|\Gamma_\rho(f_r -f ) \|^2 =\frac{1}{|G|} \int_{\partial \Omega} |r^{m_0} (\ell_\rho f\circ \theta)(r \boldsymbol{\zeta})- (\ell_\rho f\circ \theta)(\boldsymbol{\zeta}) |^2 d\nu(\boldsymbol{\zeta}) \rightarrow 0 \ \text{ as } r\rightarrow 1^-.$$

For the proof of (iv), consider $f = \sum_{j=1}^n f_j g_j $ in $\bC[\boldsymbol{p}] \bullet \bC[\boldsymbol{p}]$. Then, $\delta_r(f)(\boldsymbol{p})= \sum_{j=1}^n r^{2m_0} f_j(r^{\boldsymbol{m}} \boldsymbol{p}) g_j(r^{\boldsymbol{m}} \boldsymbol{p})$ and so,
\begin{align*}
\|\delta_r(f)\|_\bullet & \leq  \sum_{j=1}^n \|r^{m_0} f_j(r^{\boldsymbol{m}} \boldsymbol{p}) \| \|r^{m_0} g_j(r^{\boldsymbol{m}}\boldsymbol{p})\| \\
                        & = \sum_{j=1}^n \| \ell_\rho(r \boldsymbol{z}) f_j \circ \theta(r \boldsymbol{z}) \|  \| \ell_\rho(r \boldsymbol{z}) g_j \circ \theta(r \boldsymbol{z}) \| \\
                        & \leq \sum_{j=1}^n  \| \ell_\rho (f_j \circ \theta) \| \| \ell_\rho (g_j \circ \theta) \| = \sum_{j=1}^n  \| f_j\| \| g_j\|.
\end{align*}
This implies that $\|\delta_r(f)\|_\bullet \leq \|f\|_\bullet$. Therefore $\delta_r$ is contractive on $\bC[\boldsymbol{p}] \bullet \bC[\boldsymbol{p}]$ and hence it extends contractively to its completion. Also, we can write the polynomial $\delta_r(f) - f$ as $1 \cdot (\delta_r(f) - f)$ and hence $\|\delta_r(f) - f\|_\bullet \leq \|1\| \|\delta_r(f)-f\| \rightarrow 0$ as $r \rightarrow 1^{-}$. Finally, for any $g$ in $(\bC[\boldsymbol{p}] \bullet \bC[\boldsymbol{p}])_\bullet$, there exists a sequence $g_n$ in $\bC[\boldsymbol{p}] \bullet \bC[\boldsymbol{p}]$ such that $g_n \xrightarrow{\|\cdot \|_\bullet} f$. Thus, for $\epsilon >0$ choose $n$ such that $\|g_n -g\|_\bullet < \epsilon$ and for this $g_n$ there exists $0<s<1$ such $\|\delta_r(g_n)- g_n\|<\epsilon$ whenever $r \in (s,1)$. So,
\begin{align*}
\|\delta_r(g)-g\|_\bullet \leq \|\delta_r(g)-\delta_r(g_n)\|_\bullet + \|\delta_r(g_n) - g_n\|_\bullet + \|g_n -g\|_\bullet < 3\epsilon
\end{align*}
for $r \in (s,1)$. This proves the last part of statement (iv).
\end{proof}

\noindent {\em Continuation of the proof of \cref{P: Norm equality}}: 
By the discussion above, it is enough to prove that $\operatorname{Ker}J$ is trivial. Let $f \in \operatorname{Ker}(J)$. Now, approximate $f$ by a sequence $\{f_n\}_n$ in $\bC[\boldsymbol{p}] \bullet \bC[\boldsymbol{p}]$ with $\|\cdot\|_\bullet$. By continuity of $J$ and the point evaluations in $\theta(\Omega)$ we have $f_n \xrightarrow{\|\cdot\|_\odot} Jf =0$ and $f_n(\boldsymbol{p}) \rightarrow 0$ and $|f_n(\boldsymbol{p})| \leq \|S_{\rho, \theta}(\cdot, \boldsymbol{p})\| \|f_n\|_\odot$ for every $\boldsymbol{p}$ in $\theta(\Omega)$. Since $S_{\rho, \theta}(\cdot, \boldsymbol{p})$ is holomorphic in the first argument and anti-holomorphic in the second argument on $\theta(\Omega) \times \theta(\Omega)$ and continuous on $\theta(\Omega) \times \overline{\theta(\Omega)}$, the norm $\|S_{\rho, \theta}(\cdot, \boldsymbol{p})\|$ is bounded whenever $\boldsymbol{p}$ varies over a compact subset of $\theta(\Omega)$. Thus $\{f_n\}$ is bounded on every compact subset of $\theta(\Omega)$ and hence it converges to $0$ uniformly on compact subsets of $\theta(\Omega)$. 
Fix $0\leq r<1$. Note that, for each $n\in \mathbb{N}$, $\delta_r(f_n)(\theta(\boldsymbol{z}))= r^{2m_0} f_n \circ \theta (r \boldsymbol{z}) $ which is holomorphic in a neighbourhood of $\overline{\theta(\Omega)}$ and so, $\delta_r(f_n)$ converges to $0$ uniformly on $\overline{\theta(\Omega)}$. Therefore, $\|\delta_r(f_n)\| \rightarrow 0$ as $ n \rightarrow \infty$. 
Finally the following estimate,
\begin{align*}
 \|\delta_r(f) \|_\bullet \leq \|\delta_r(f) -\delta_r(f_n)\|_\bullet + \|\delta_r(f_n)\|_\bullet \leq \|f-f_n\|_\bullet + \|\delta_r(f_n)\| 
\end{align*}
shows that $\delta_r(f) =0$ for each $r$. Hence, $f=0$.
\end{proof}

Let $\textup{s-Hank}(\partial \theta(\Omega))$ be the set of all $L^2(\partial \theta(\Omega), \mu_{\rho, \theta})$ symbols $\varphi$ for which $h_\varphi$ is bounded. Then for $\varphi \in H^2_\rho (\theta(\Omega)) \cap \textup{s-Hank}(\partial \theta(\Omega)) $ we define $\|\varphi\|_{\textup{s-Hank}(\partial \theta(\Omega))}= \|h_\varphi\|$. This makes $H^2_\rho (\theta(\Omega)) \cap \textup{s-Hank}(\partial \theta(\Omega))$ a normed linear space.

\begin{thm}\label{dual-tensor-product}
 The space $H^2_\rho (\theta(\Omega)) \cap \textup{s-Hank}(\partial \theta(\Omega))$ is isometrically isomorphic to the continuous dual $\left(H_\rho^2(\theta(\Omega)) \odot H_\rho^2(\theta(\Omega)) \right)^*$ of the weak product space via the following map: $ \varphi \mapsto L_\varphi$, where $L_\varphi (f)= \int_{\partial \theta(\Omega)} f \overline{\varphi} d\mu_{\rho, \theta}$ for every $f \in H_\rho^2(\theta(\Omega)) \odot H_\rho^2(\theta(\Omega))$.
\end{thm}

\begin{proof}
Let $\varphi$ be in $H^2_\rho (\theta(\Omega)) \cap \textup{s-Hank}(\partial \theta(\Omega))$. Consider an arbitrary $F$ in $ \bC[\boldsymbol{p}] \bullet \bC[\boldsymbol{p}]$ of the form $F=\sum_{j=1}^n f_j g_j$ where $g_j, h_j \in \bC[\boldsymbol{p}]$. 
A short calculation shows that $$\sum_{j=1}^n \langle h_\varphi (g_j), h_j\rangle_{H_\rho^2(\theta(\Omega))} = \int_{\partial \theta(\Omega)} \varphi \bar{F} d\mu_{\rho, \theta}.$$ Since $h_\varphi$ is bounded, 
$$|\int_{\partial \theta(\Omega)} \bar{\varphi} F d\mu_{\rho, \theta}|= |\int_{\partial \theta(\Omega)} \varphi \bar{F} d\mu_{\rho, \theta} | \leq \|h_\varphi \| \sum_{j=1}^n\|g_j\|_{H_\rho^2(\theta(\Omega))}\|h_j\|_{H_\rho^2(\theta(\Omega))}.$$ 
This holds true for any representation $F=\sum_{j=1}^n g_j h_j$. Using \cref{P: Norm equality}, we have 
$$|\int_{\partial \theta(\Omega)} \bar{\varphi} F d\mu_{\rho, \theta} | \leq \|h_\varphi\| \|F\|_\bullet  =\|h_\varphi\| \|F\|_\odot.$$ 
Thus, $L_\varphi$ is a bounded linear functional on the dense subspace $\bC[\boldsymbol{p}] \bullet \bC[\boldsymbol{p}] $ and hence extends continuously to $H_\rho^2(\theta(\Omega)) \odot H_\rho^2(\theta(\Omega)) $. Moreover, $\|L_\varphi\| \leq \|h_\varphi\|$.

Conversely, suppose that $\chi$ is a bounded linear functional on $H_\rho^2(\theta(\Omega)) \odot H_\rho^2(\theta(\Omega))$. For $f \in H_\rho^2(\theta(\Omega)) $, we can write $f= f\cdot 1 $  and so, $f \in H_\rho^2(\theta(\Omega)) \odot H_\rho^2(\theta(\Omega)) $ with $\|f\|_\odot \leq \|f\|$. Therefore, we have 
$$ \| \chi(f)\| \leq \|\chi\|\|f\|_\odot \leq \|\chi \| \|f\|.$$ Thus $\chi$ defines a bounded linear functional on $H_\rho^2(\theta(\Omega))$ as well and hence there exists a $\varphi \in H_\rho^2(\theta(\Omega))$ such that $\chi(f) = \langle f, \varphi \rangle$ that is, $\chi = L_\varphi$. 

If $g \in H_\rho^2(\theta(\Omega))$ then
$$|\langle h_\varphi(f), g\rangle| =|\int_{\partial \theta(\Omega)} fg \bar{\varphi} d\mu_{\rho, \theta}| \leq \|L_\varphi \| \|fg\|_\odot  \leq \| L_\varphi \| \|f\|_{H_\rho^2(\theta(\Omega))} \|g\|_{H_\rho^2(\theta(\Omega))}.$$ Therefore $h_\varphi$ is bounded and $\|h_\varphi\| \leq \|L_\varphi \|$. This completes the proof.
\end{proof}

\begin{remark} \label{R: BMO} \leavevmode
\begin{itemize}
\item[(i)]
We can think of $\varphi \in H^2_\rho (\theta(\Omega)) \cap \textup{s-Hank}(\partial \theta(\Omega)) $ as an element of the dual of the weak product space with norm $\|\varphi\|_* = \|L_\varphi\|$ via the isometry in the theorem above. We shall use this identification below.

\item[(ii)] In case of the disc, the Euclidean ball and more generally smoothly bounded pseudoconvex domains, the dual of the corresponding weak product spaces are the same as the space of analytic functions with bounded mean oscillation ($\textup{BMO}$) defined appropriately. The space of $\textup{BMO}$ functions  is one of the most interesting objects in real-variable Hardy space theory. Looking at the analogy in these concrete cases we can view the dual of the weak product space as an abstract substitute for the analytic $\textup{BMO}$ ($\textup{BMOA}$) space in our set-up of quotient domains.   
\end{itemize}
\end{remark}

 A classical theorem of Hartman characterizes the compact small Hankel operators on $H^2(\mathbb{D})$ \cite{Hartman}. It says that a Hankel operator $H$ is compact if and only if there exists a continuous function $\varphi$ on $\mathbb{T}$ such that $H$ is the same as the Hankel operator corresponding to the symbol $\varphi$. In fact, compact Hankel operators on $H^2(\mathbb{D})$ are operator-norm limit of Hankel operators of finite rank. Such a characterization of compact Hankel operators is also known for compact Hankel operators on the Hardy space of the Euclidean ball $H^2\left( \mathbb{B}_d \right)$; see \cite{CRW}. The following lemma will be useful for approximation of compact small Hankel operators in our perspective.
\begin{lemma}\label{Hart-L1}
If $\psi $ is a polynomial in $\boldsymbol{p}$ and $\bar{\boldsymbol{p}}$ and $\varphi = P_{H_\rho^2(\theta(\Omega))}(\psi)$, then $h_\varphi$ is a finite rank operator.
\end{lemma}
\begin{proof}
Suppose, $\psi = \sum_{|\boldsymbol{\alpha}|\leq M_1, \ |\boldsymbol{\beta}|\leq M_2 } a_{\boldsymbol{\alpha}, \boldsymbol{\beta}}\boldsymbol{p}^{\boldsymbol{\alpha}} \bar{\boldsymbol{p}}^{\boldsymbol{\beta}} $. Then $\varphi$ will be polynomial, say, $\varphi = \sum_{|\boldsymbol{\gamma}| \leq N} a_{\boldsymbol{\gamma}} \boldsymbol{p}^{\boldsymbol{\gamma}} $. Note that,
$$h_\varphi = \sum_{|\boldsymbol{\gamma}|\leq N} a_{\boldsymbol{\gamma}} h_{\boldsymbol{p}^{\boldsymbol{\gamma}}} .$$ So, it is enough to show that $h_{\boldsymbol{p}^{\boldsymbol{\gamma}}}$ is finite rank. To that end, let us note with the unitary $\Gamma_\rho$ as in \eqref{Gamma_rho} and $g\in \bC[\boldsymbol{p}]$ that 
$$ \Gamma_\rho h_{\boldsymbol{p}^{\boldsymbol{\gamma}}} (g) = \Gamma_\rho P_{H_\rho^2(\theta(\Omega))} \left( \boldsymbol{p}^{\boldsymbol{\gamma}} \bar{g} \right) =  P_{R_\rho ^G H^2(\Omega)} \Gamma_\rho' \left( \boldsymbol{p}^{\boldsymbol{\gamma}} \bar{g} \right) = \frac{1}{\sqrt{|G|}}P_{R_\rho ^G H^2(\Omega)} \left(\ell_\rho \theta ^{\boldsymbol{\gamma}} \overline{g \circ \theta} \right) .$$
For $\boldsymbol{\gamma} =(\gamma_1, \dots, \gamma_d)$, $\theta^{\boldsymbol{\gamma}}= \theta_1^{\gamma_1}\cdots \theta_d^{\gamma_d}$ and so,  the degree of the homogeneous polynomial $\ell_\rho \theta ^{\boldsymbol{\gamma}}$ is $m_0 + \sum_{j=1}^d \gamma_j m_j$. So, for every homogeneous $G$-invariant polynomial $g \circ \theta $ the polynomial $P_{R_\rho ^G H^2(\Omega)} \left(\ell_\rho \theta ^{\boldsymbol{\gamma}} \overline{g \circ \theta} \right)$ is of degree atmost $d_0= m_0 + \sum_{j=1}^d \gamma_j m_j$. Since range of $P_{R_\rho^G H^2(\Omega)}$ is contained in the range of $P_{ H^2(\Omega)}$, it is enough to prove $P_{H^2(\Omega)} \left(\ell_\rho \theta ^{\boldsymbol{\gamma}} \overline{g \circ \theta} \right)$ is of degree atmost $d_0$.  Indeed,
\begin{align*}
&P_{ H^2(\Omega)} \left(\ell_\rho \theta ^{\boldsymbol{\gamma}} \overline{(g \circ \theta)} \right)(\boldsymbol{z})= \int_{\partial\Omega} \ell_\rho(\boldsymbol{\zeta}) \theta ^{\boldsymbol{\gamma}}(\boldsymbol{\zeta}) \overline{(g \circ \theta) (\boldsymbol{\zeta})} \overline{S_\Omega(\boldsymbol{\zeta}, \boldsymbol{z})} d\nu(\boldsymbol{\zeta})\\
&= \sum_{k=0}^\infty \sum_{\alpha=0}^{N_k} \int_{\partial \Omega} \ell_\rho(\boldsymbol{\zeta}) \theta ^{\boldsymbol{\gamma}}(\boldsymbol{\zeta})\overline{(g \circ \theta) (\boldsymbol{\zeta})} \overline{e_{\alpha}^{(k)}(\boldsymbol{\zeta})} e_{\alpha}^{(k)}(\boldsymbol{z})d\nu(\boldsymbol{\zeta})\\
&=\sum_{k=0}^\infty \sum_{\alpha=0}^{N_k} e_{\alpha}^{(k)}(\boldsymbol{z}) \langle \ell_\rho \theta ^{\boldsymbol{\gamma}}, (g \circ \theta ) e_{\alpha}^{(k)} \rangle \\
&=\sum_{k=0}^{d_0} \sum_{\alpha=0}^{N_k} e_{\alpha}^{(k)}(\boldsymbol{z}) \langle \ell_\rho \theta ^{\boldsymbol{\gamma}}, (g \circ \theta ) e_{\alpha}^{(k)} \rangle.
\end{align*}
 This is a polynomial of degree atmost $d_0$. Therefore, $\Gamma_\rho h_{\boldsymbol{p}^{\boldsymbol{\gamma}}}$ and hence $ h_{\boldsymbol{p}^{\boldsymbol{\gamma}}}$ is finite rank operator. This completes the proof.
\end{proof}

In the classical case, the analytic symbols that define bounded compact Hankel operators are functions with vanishing mean oscillations (\textup{VMOA}). Also, the holomorphic polynomials are dense in $\textup{VMOA}$ with respect to $\textup{BMO}$-norm. In our setup, as discussed in \cref{R: BMO} (ii), we think of the dual of the weak product space as an abstract generalization of \textup{BMOA}. Hence the following definition is quite natural.
\begin{definition}
Let $\operatorname{VMOA}(\partial \theta(\Omega))$ denote the closure of holomorphic polynomials in $(H_\rho^2(\theta(\Omega)) \odot H_\rho^2(\theta(\Omega)) )^*$.
\end{definition}
Now we are in a position to prove an analogue of Hartman's theorem in our case. The central idea of the proof is to approximate compact small Hankel operators by small Hankel operators corresponding to some polynomial symbols. We shall also use the properties of $\delta_r$ from \cref{Hart-L2}.

\begin{thm} \label{Hartman}
For $\varphi \in H_\rho^2(\theta(\Omega)) $, the small Hankel operator $h_\varphi$ is a compact operator if and only if $\varphi \in \operatorname{VMOA}(\partial \theta(\Omega))$.
\end{thm}
\begin{proof}
Let $h_\varphi$ be a compact small Hankel operator. Since $\delta_r \xlongrightarrow{r\to 1^-} \operatorname{Id}_{H_\rho^2(\theta(\Omega))}$ in strong operator topology and $h_\varphi$ is compact, $\delta_r h_\varphi \xlongrightarrow{\|\cdot\|_{op}} h_\varphi$. Similarly, self-adjointness of $\delta_r$ and compactness of $h_\varphi^*$ imply $ h_\varphi \delta_r \xlongrightarrow{\|\cdot\|_{op}} h_\varphi$. Now 
\begin{align*}
\| h_\varphi - r^{-2m_0} \delta_r h_\varphi \delta_r \| & \leq \| h_\varphi - r^{-2m_0} \delta_r h_\varphi \| + \| r^{-2m_0}\delta_r h_\varphi - r^{-2m_0} \delta_r h_\varphi \delta_r\| \rightarrow 0                                                                                        
\end{align*}
as $r\rightarrow 1^-$ with $r>0$.
So, $h_\varphi$ can be approximated by $r^{-2m_0} \delta_r h_\varphi \delta_r $. Now our aim is to approximate $r^{-2m_0} \delta_r h_\varphi \delta_r $ by finite rank operators. For any two holomorphic polynomials $f$ and $g$ we  have,
$$ \left \langle h_{\delta_r(\varphi)}(f), g  \right \rangle = \int_{\partial \theta(\Omega)} r^{2m_0} \varphi(r^{m_1}p_1,\dots, r^{m_d}p_d ) \overline{f g} d\mu_{\rho,\theta} = \left \langle \delta_r(\varphi), f g  \right \rangle =\left \langle  \varphi, \delta_r (f g) \right \rangle $$
and
\begin{align*}
&\left \langle r^{-2m_0} \delta_r h_\varphi \delta_r (f), g \right \rangle = \frac{1}{r^{2m_0}} \left \langle h_\varphi(\delta_r f),\delta_r(g) \right \rangle \\
& = \int_{\partial \theta(\Omega)} \frac{r^{2m_0} r^{2m_0}}{r^{2m_0}} \varphi(p_1,\dots, p_d)  \overline{(f g)}(r^{m_1}p_1,\dots, r^{m_d}p_d ) d\mu_{\rho,\theta} = \langle \varphi, \delta_r(f g) \rangle.
\end{align*}
Therefore, $h_{\delta_r(\varphi)} = r^{-2m_0}\delta_r h_\varphi \delta_r $. From \cref{dual-tensor-product}, we have $\|L_{\varphi_r - \varphi}\| = \|h_{\varphi_r - \varphi}\|$. Also, $\delta_r(\varphi)$ being holomorphic in a neighbourhood of $\overline{\theta(\Omega)}$, we can approximate $\delta_r(\varphi)$ by a sequence of holomorphic polynomials in $\overline{\theta(\Omega)}$ uniformly. Here we are using the fact that $\overline{\theta(\Omega)}$ is polynomially convex and therefore the Oka-Weil theorem ensures that such a polynomial approximation can be applied.

Choose $\epsilon >0$. Then there exists $r \in (0,1)$ such that $\|h_{\varphi -\delta_r(\varphi)}\| < \epsilon$. For this $\delta_r(\varphi)$, choose a holomorphic polynomial $f$ so that $\|\delta_r(\varphi) - f\|_{\infty, \overline{\theta(\Omega)}} < \epsilon$. But $\|h_{\delta_r(\varphi) - f} \| \leq  \|\delta_r(\varphi) -f\|_{\infty, \overline{\theta(\Omega)}} <\epsilon$. So, $\|h_\varphi - h_f \| < 2\epsilon$. Finally, using the norm equivalence of the Hankel operator and the dual norm of the corresponding symbol we get $\|\varphi - f \|_* < C\epsilon$, for some absolute constant $C$. Therefore $\varphi \in \operatorname{VMOA}(\partial\theta(\Omega))$. 

To prove the converse, assume $\varphi \in \operatorname{VMOA}(\partial\theta(\Omega))$. Let $\epsilon >0$. Then there exists a holomorphic polynomial $\psi$ so that $\|\varphi - \psi \|_* < \epsilon $. Again by the norm equivalence as earlier, $\|h_{\varphi - \psi} \| < C\epsilon$ for some absolute constant $C$. Now, \cref{Hart-L1} implies $h_\psi $ is a finite rank operator. So, $h_\varphi$ can be approximated by finite rank operators in operator norm and hence it is compact. This completes the proof.
\end{proof}


\section{The big Hankel operator} \label{Big-Hankel}
\begin{definition}
For $\varphi \in H_\rho^2(\theta(\Omega))^{\perp}$, the \textit{big Hankel operator} $H_\varphi$ is a densely defined linear operator from $H_\rho^2(\theta(\Omega)) $ to $H_\rho^2(\theta(\Omega))^{\perp}$ given by $$ H_\varphi(f) =(I- P_{H_\rho^2(\theta(\Omega))})(\varphi f), \text{ for all holomorphic polynomial } f \in H_\rho^2(\theta(\Omega)),$$ where $P_{H_\rho^2(\theta(\Omega))}$ is the orthogonal projection of $L^2(\partial \theta(\Omega),\mu_{\rho, \theta})$ onto $H_\rho^2(\theta(\Omega))$.
\end{definition}
The big Hankel operators are well studied in $H^2\left( \mathbb{D}^d\right)$; see for example \cite{Fe-Sa, Russo}. Let $S_{\rho, \theta}$ be the reproducing kernel (or, the {\em Szeg\"o kernel}) for $H_\rho^2(\theta(\Omega))$. Then the {\em Poisson-Szeg\"o kernel} is given by
$$
\mathcal{P}_\rho (\boldsymbol{p},\boldsymbol{\zeta}) = \frac{|S_{\rho, \theta}(\boldsymbol{p}, \boldsymbol{\zeta})|^2}{S_{\rho, \theta}(\boldsymbol{p}, \boldsymbol{p})} \text{ for } (\boldsymbol{p},\boldsymbol{\zeta})\in \theta(\Omega) \times \partial \theta(\Omega).
$$ 
Note that one can write $\mathcal{P}_\rho $ in terms of the normalized kernel functions as follows. 
\noindent Let $s_{\boldsymbol{p}} (\boldsymbol{w})= S_{\rho, \theta} (\boldsymbol{w}, \boldsymbol{p})$ and $\hat{s}_{\boldsymbol{p}} = \frac{s_{\boldsymbol{p}}}{\|s_{\boldsymbol{p}} \|}$. Then $\mathcal{P}_\rho (\boldsymbol{p}, \boldsymbol{\zeta}) = \left| \hat{s}_{\boldsymbol{p}}(\boldsymbol{\zeta}) \right|^2 $.
For a function $f \in L^2(\partial \theta(\Omega), \mu_{\rho, \theta})$, we define the Poisson-Szeg\"o extension $\tilde{f}$ of $f$ by
$$
\tilde{f}(\boldsymbol{p}) = \int_{\partial \theta(\Omega)} f(\boldsymbol{\zeta}) \mathcal{P}_\rho (\boldsymbol{p},\boldsymbol{\zeta}) d\mu_{\rho, \theta}(\boldsymbol{\zeta})
$$
for $\boldsymbol{p}\in \theta(\Omega)$. Motivated by Ahern and Youssfi in \cite{Ah-You}, we make the following definition. 

\begin{definition}
$\operatorname{BMO}(\partial \theta(\Omega))$ is defined  to be the collection of functions $f$ in $L^2_\rho(\partial \theta(\Omega))$ such that 
\begin{align}
\|f\|_{*} = \sup_{\boldsymbol{p} \in \theta(\Omega)} \int_{\partial \theta(\Omega)} |f(\boldsymbol{\zeta})-\tilde{f}(\boldsymbol{p})|^2 \mathcal{P}_\rho (\boldsymbol{p},\boldsymbol{\zeta} ) d\mu_{\rho, \theta}(\boldsymbol{\zeta})<\infty.
\end{align}
\end{definition}
A word of caution: in the case of the polydisc, this is different from the Chang-Fefferman BMO \cite{Chang-Fefferman, Cotlar-Sadosky-IEOT}. 
 We give a necessary condition for simultaneous boundedness of $H_f$ and $H_{\bar{f}}$ for a symbol $f$ in terms of the BMO defined above. The motivation is from Theorem B of \cite{Ah-You}. The idea of the proof is inspired from \cite{Beatrous-Li, Zhu-book}.
 
\begin{thm} \label{fInBMO}
Let $f \in L^2(\partial \theta(\Omega), \mu_{\rho, \theta})$. Assume that $\frac{1}{S_{\rho, \theta}(\cdot, \boldsymbol{p})}$ in $\cA(\theta(\Omega))$ for each $\boldsymbol{p}$ in $\theta(\Omega)$. If the big Hankel operators $H_f$ and $H_{\bar{f}}$ are bounded operators, then $f \in \operatorname{BMO}(\partial \theta(\Omega))$.
\end{thm}

\begin{proof}
Let $\boldsymbol{p} \in \theta(\Omega)$. Then
\begin{align*}
&\int_{\partial \theta(\Omega)} |f(\boldsymbol{\zeta})-\tilde{f}(\boldsymbol{p})|^2 \mathcal{P}_\rho (\boldsymbol{p},\boldsymbol{\zeta} ) d\mu_{\rho, \theta}(\boldsymbol{\zeta}) \\
&= \int_{\partial \theta(\Omega)} \left[ |f(\boldsymbol{\zeta})|^2 + |\tilde{f}(\boldsymbol{p})|^2  - 2\operatorname{Re} \left(f(\boldsymbol{\zeta}) \overline{\tilde{f}(\boldsymbol{p})}\right)  \right] \mathcal{P}_\rho (\boldsymbol{p},\boldsymbol{\zeta} )  d\mu_{\rho, \theta}(\boldsymbol{\zeta}) \\
&= \int_{\partial \theta(\Omega)} \left| f(\boldsymbol{\zeta}) \hat{s}_{\boldsymbol{p}}(\boldsymbol{\zeta}) \right|^2 d\mu_{\rho, \theta}(\boldsymbol{\zeta}) + |\tilde{f}(\boldsymbol{p})|^2 - 2 \overline{\tilde{f}(\boldsymbol{p})} \tilde{f}(\boldsymbol{p}) \\
&= \|f \hat{s}_{\boldsymbol{p}} \|^2 - |\tilde{f}(\boldsymbol{p})|^2.
\end{align*}
Again, $\tilde{f}(\boldsymbol{p})$ can be written as 
\begin{align*}
\tilde{f}(\boldsymbol{p}) = \int_{\partial \theta(\Omega)} f(\boldsymbol{\zeta}) \hat{s}_{\boldsymbol{p}}(\boldsymbol{\zeta}) \frac{S(\boldsymbol{p} ,\boldsymbol{\zeta})}{\| s_{\boldsymbol{p}}\|} d\mu_{\rho, \theta}(\boldsymbol{\zeta}) 
                          = \frac{1}{\| s_{\boldsymbol{p}}\|} \langle f \hat{s}_{\boldsymbol{p}}, s_{\boldsymbol{p}} \rangle 
                          = \frac{1}{\| s_{\boldsymbol{p}}\|}   P_{H_\rho^2(\theta(\Omega))} (f \hat{s}_{\boldsymbol{p}} )(\boldsymbol{p}).
\end{align*}
Similarly,
 \begin{align*}
\overline{\tilde{f}(\boldsymbol{p})} = \frac{1}{\| s_{\boldsymbol{p}}\|}   P_{H_\rho^2(\theta(\Omega))} (\bar{f} \hat{s}_{\boldsymbol{p}} )(\boldsymbol{p}).
\end{align*}
Therefore,
\begin{align*}
|\tilde{f}(\boldsymbol{p})|^2 = \frac{1}{\| s_{\boldsymbol{p}}\|^2} P_{H_\rho^2(\theta(\Omega))} (f \hat{s}_{\boldsymbol{p}} )(\boldsymbol{p})    P_{H_\rho^2(\theta(\Omega))} (\bar{f} \hat{s}_{\boldsymbol{p}} ) (\boldsymbol{p}).
\end{align*}
Further,

\noindent $\|f  \hat{s}_{\boldsymbol{p}} \|^2 = \|P_{H_\rho^2(\theta(\Omega))} (f \hat{s}_{\boldsymbol{p}} )\|^2 + \|H_f(\hat{s}_{\boldsymbol{p}})\|^2$ and 
 $\|\bar{f}  \hat{s}_{\boldsymbol{p}} \|^2 = \|P_{H_\rho^2(\theta(\Omega))} (\bar{f} \hat{s}_{\boldsymbol{p}} )\|^2 + \|H_{\bar{f}}(\hat{s}_{\boldsymbol{p}})\|^2$. Using all these expressions we have,
 \begin{align*}
 & \|H_f(\hat{s}_{\boldsymbol{p}})\|^2 + \|H_{\bar{f}}(\hat{s}_{\boldsymbol{p}})\|^2 - \|f \hat{s}_{\boldsymbol{p}} \|^2 + |\tilde{f}(\boldsymbol{p})|^2 \\
 &= \|\bar{f}  \hat{s}_{\boldsymbol{p}} \|^2 - \|P_{H_\rho^2(\theta(\Omega))} (f \hat{s}_{\boldsymbol{p}} )\|^2 - \|P_{H_\rho^2(\theta(\Omega))} (\bar{f} \hat{s}_{\boldsymbol{p}} )\|^2 + |\tilde{f}(\boldsymbol{p})|^2  \\
 &= \langle \bar{f} \hat{s}_{\boldsymbol{p}} , \bar{f} \hat{s}_{\boldsymbol{p}}  \rangle - \langle P_{H_\rho^2(\theta(\Omega))} (f \hat{s}_{\boldsymbol{p}} ),  f \hat{s}_{\boldsymbol{p}}  \rangle - \langle P_{H_\rho^2(\theta(\Omega))} (\bar{f} \hat{s}_{\boldsymbol{p}} ),  \bar{f} \hat{s}_{\boldsymbol{p}} \rangle + |\tilde{f}(\boldsymbol{p})|^2 \\
 &= \int_{\partial \theta(\Omega)} \left[ |f|^2 \hat{s}_{\boldsymbol{p}}^2 - \bar{f} \hat{s}_{\boldsymbol{p}} P_{H_\rho^2(\theta(\Omega))} (f \hat{s}_{\boldsymbol{p}} ) - f \hat{s}_{\boldsymbol{p}} P_{H_\rho^2(\theta(\Omega))} (\bar{f} \hat{s}_{\boldsymbol{p}} ) \right](\boldsymbol{\zeta}) \frac{\overline{s_{\boldsymbol{p}}}(\boldsymbol{\zeta})}{s_{\boldsymbol{p}}(\boldsymbol{\zeta})}  d\mu_{\rho,\theta}(\boldsymbol{\zeta})\\
 & \ + \int_{\partial \theta(\Omega)}    P_{H_\rho^2(\theta(\Omega))} (f \hat{s}_{\boldsymbol{p}} )(\boldsymbol{\zeta}) P_{H_\rho^2(\theta(\Omega))} (\bar{f} \hat{s}_{\boldsymbol{p}} )(\boldsymbol{\zeta})      \frac{\overline{s_{\boldsymbol{p}}}(\boldsymbol{\zeta})}{s_{\boldsymbol{p}}(\boldsymbol{\zeta})}  d\mu_{\rho,\theta}(\boldsymbol{\zeta}) \\
 & = \int_{\partial \theta(\Omega)} (I-P_{H_\rho^2(\theta(\Omega))})(f \hat{s}_{\boldsymbol{p}})(\boldsymbol{\zeta}) (I-P_{H_\rho^2(\theta(\Omega))})(\bar{f} \hat{s}_{\boldsymbol{p}}) (\boldsymbol{\zeta}) \frac{\overline{s_{\boldsymbol{p}}}(\boldsymbol{\zeta})}{s_{\boldsymbol{p}}(\boldsymbol{\zeta})}  d\mu_{\rho,\theta}(\boldsymbol{\zeta}) \\
 &= \left \langle H_f(\hat{s}_{\boldsymbol{p}}), \overline{H_{\bar{f}} (\hat{s}_{\boldsymbol{p}})} \frac{s_{\boldsymbol{p}}}{\overline{s_{\boldsymbol{p}}}} \right \rangle.
 \end{align*}
 Therefore,
 \begin{align*}
 \|f \hat{s}_{\boldsymbol{p}} \|^2 - |\tilde{f}(\boldsymbol{p})|^2 &= \|H_f(\hat{s}_{\boldsymbol{p}})\|^2 + \|H_{\bar{f}}(\hat{s}_{\boldsymbol{p}})\|^2 - \left \langle H_f(\hat{s}_{\boldsymbol{p}}), \overline{H_{\bar{f}} (\hat{s}_{\boldsymbol{p}})} \frac{s_{\boldsymbol{p}}}{\overline{s_{\boldsymbol{p}}}} \right \rangle \\
 & \leq  \|H_f(\hat{s}_{\boldsymbol{p}})\|^2 + \|H_{\bar{f}}(\hat{s}_{\boldsymbol{p}})\|^2 + \|H_f(\hat{s}_{\boldsymbol{p}})\|  \|H_{\bar{f}}(\hat{s}_{\boldsymbol{p}})\| \\
 &\leq 2 (\|H_f\|^2 + \| H_{\bar{f}}\|^2 ).
 \end{align*}
 This estimate shows that $f \in BMO(\partial \theta(\Omega))$.
\end{proof}

 One would like to have a more geometric definition of the $\operatorname{BMO}(\partial \theta(\Omega))$ space so that one can obtain more information about the symbol $f$ for which $H_f$ and $H_{\bar{f}}$ are bounded. No such criterion is known yet.
\begin{corollary}
Suppose $S_{\rho, \theta}(\boldsymbol{p}, \boldsymbol{p}) \rightarrow \infty $ as $\theta(\Omega) \ni \boldsymbol{p} \rightarrow \partial \theta(\Omega)$. If $H_f$ and $H_{\bar{f}}$ are compact operators then 
\begin{align*}
\int_{\partial \theta(\Omega)} |f(\boldsymbol{\zeta})-\tilde{f}(\boldsymbol{p})|^2 \mathcal{P}_\rho (\boldsymbol{p}, \boldsymbol{\zeta}) d\mu_{\rho, \theta}(\boldsymbol{\zeta}) \rightarrow 0 \ \text{ as } \boldsymbol{p} \rightarrow \partial \theta(\Omega).
\end{align*}
\end{corollary}
\begin{proof}
  Also, for any $g \in H_\rho^2(\theta(\Omega))$,
\begin{align*}
\langle g, \hat{s}_{\boldsymbol{p}} \rangle = \frac{g(\boldsymbol{p})}{\sqrt{S_{\rho, \theta}(\boldsymbol{p}, \boldsymbol{p})}} \rightarrow 0
\end{align*}
as $\boldsymbol{p} \rightarrow \partial \theta(\Omega)$. So, $\hat{s}_{\boldsymbol{p}} \rightarrow 0$ weakly, as $\boldsymbol{p} \rightarrow \partial \theta(\Omega)$. Hence, by compactness $\|H_f(\hat{s}_{\boldsymbol{p}})\|$ and $ \|H_{\bar{f}}(\hat{s}_{\boldsymbol{p}})\|$ both tend to zero as $\boldsymbol{p} \rightarrow \partial \theta(\Omega)$.  The estimate at the end of the theorem above concludes the proof.
\end{proof}
The assumption about the singularity of $S_{\rho, \theta}(\cdot, \cdot)$ on the boundary diagonal is not very restrictive. It is known that the Szeg\"o kernels of the standard Hardy spaces of smoothly bounded strongly pseudoconvex domains have such singularities \cite{Hirachi-Annals}. The Szeg\"o kernels of $\bD^d$ and the quotient domain $\bD^d/ S_d$ have this property as can be seen from the explicit expressions of the Szeg\"o kernel.

With the help of the following result from \cite{Ba-Ti}, we now investigate whether Nehari's theorem holds for big Hankel operators on the quotient domains. The failure in the case of $H^2\left(\mathbb{T}^2\right)$ is so spectacular that it has intrigued many, and consequently different proofs  emerged in \cite{Ah-You, Fe-Sa}. 
\begin{thm} \label{Bakonyi-Timotin}
The big Hankel operator $H_\varphi $ on $H^2(\bD^2)$ corresponding to the symbol 
$$\varphi (z_1, z_2)= \sum_{n=1}^\infty \frac{1}{n} z_1^n \bar{z}_2^n $$
is a bounded operator. However, there does not exist a holomorphic function $g$ on $\bD^2$ such that $\varphi + g \in L^\infty (\bT^2)$.
\end{thm}
 The following theorem produces a counter example to  Nehari's theorem for big Hankel operators.
\begin{thm}\label{Failure-Nehari}
Consider the polydisc $\bD^d$ with a pseudo-reflection group $G$ acting on it as a pseudo-reflection group. There is a function $\tilde{\Phi}$ on the \v Silov boundary $\partial \theta(\bD^d)$ (which is known to be the image of the torus $\bT^d$ under $\theta$) such that the big Hankel operator $H_{\tilde{\Phi}}$ is a bounded operator. But there is no holomorphic function $g$ on $\theta(\bD^d)$ such that $\tilde{\Phi} + g$ is in $L ^\infty (\partial \theta(\bD^d), \mu_{\rho, \theta})$.
\end{thm}

\begin{proof}
To produce the function $\tilde{\Phi}$, consider the function $\varphi$ as in \cref{Bakonyi-Timotin} and define
$$f(z_1, \dots, z_d)= \varphi(z_1, z_2).  $$ 

Any member $\sigma$ of the pseudo-reflection group $G$ is a automorphism of $\bD^d$ which keeps the origin fixed and hence there exist a permutation $\tilde{\sigma}$ in $S_d$ and $d$ scalars $\alpha_{1\sigma}, \dots, \alpha_{d \sigma} \in \bT$ such that  
\begin{align}\label{Eq: Phi}
\sigma (\boldsymbol{z})= \left( \alpha_{1\sigma} z_{\tilde{\sigma}(1)}, \dots, \alpha_{d\sigma} z_{\tilde{\sigma}(d)} \right).
\end{align}
Then, 
$$ f \circ \sigma (\boldsymbol{z})= \sum_{n\in \mathbb{N}} \frac{1}{n} \alpha_{1\sigma}^n \bar{\alpha}_{2\sigma}^n z_{\tilde{\sigma}(1)}^n \bar{z}_{\tilde{\sigma}(2)}^n .$$
Define
$$
\Phi = \frac{1}{|G|} \sum_{\sigma \in G} f \circ \sigma.
$$

We claim that the big Hankel operator $H_{f \circ \sigma}$ is bounded on $H^2(\mathbb{T}^d)$ for each $\sigma \in G$. To that end, assume that $\sigma$ is as above. Consider the unitary map $U_\sigma$ from $L^2(\bT^d)$ onto itself defined as
$$
U_\sigma (\psi)= \psi \circ \sigma \text{ for } \psi \in L^2(\bT^d).
$$
Also, $U_\sigma$ reduces $H^2(\bD^d)$. We take a holomorphic polynomial of the form $g \circ \sigma$. Then
\begin{align*}
H_{f \circ \sigma}(g \circ \sigma) = \left( I- P_{H^2(\bD^d)}\right) (f \circ \sigma \  g\circ \sigma)
\end{align*} 
It can be shown using the map $U_\sigma$ that 
\begin{align*}
\|\left( I- P_{H^2(\bD^d)} \right)  (f g)\circ \sigma \| = \|\left( I- P_{H^2(\bD^d)} \right) (f g) \| \leq \|H_{f}\| \|g\|.
\end{align*}
Therefore 
\begin{align*}
\| H_{f \circ \sigma}(g \circ \sigma)\| \leq \| H_{f}\| \|g \circ \sigma \| \ (\text{ since } U_\sigma \text{ is unitary}).
\end{align*}
This proves the claim. So, the map
$$
H_{\Phi} = \frac{1}{|G|} \sum_{\sigma \in G} H_{f \circ \sigma}
$$
defines a bounded big Hankel operator on $H^2(\bT^d)$. In fact, $\Phi $ is a $G$-invariant function and hence there is a $\tilde{\Phi}$ on $\partial \theta(\bD^d)$ such that $\Phi = \tilde{\Phi} \circ \theta $.
Now, \begin{align*}
\Gamma_\rho H_{\tilde{\Phi}}(\psi) &= \Gamma_\rho (\tilde{\Phi} \psi) - \Gamma_\rho P_{H_\rho ^2 (\theta(\bD^d))} (\tilde{\Phi} \psi)\\
                                           &= \frac{\ell_\rho}{\sqrt{|G|}} (\tilde{\Phi} \psi)\circ \theta - P_{R_\rho ^G H^2(\bD^d)} \Gamma_\rho (\tilde{\Phi} \psi)\\
                                           &= \left( I- P_{H^2(\bD^d)} \right) \Phi \Gamma_\rho (\psi)= H_{\Phi}\Gamma_\rho (\psi).
\end{align*}

Since $H_{\Phi}$ is a bounded operator and $\Gamma_\rho $ is a unitary, $H_{\tilde{\Phi}}$ is a bounded big Hankel operator on $H_\rho^2(\theta(\Omega))$.

 To prove the remaining part of the theorem, assume on the contrary that there exists a holomorphic function $g$ such that $\Psi = \tilde{\Phi} + g$ is in $L^\infty (\partial \theta(\bD^d), \mu_{\rho, \theta})$. Then we have $$\Psi \circ \theta = \tilde{\Phi} \circ \theta + g\circ \theta = \Phi + g\circ \theta .$$
Also, $\Psi \circ \theta $ is in $L^\infty (\bT^d)$. This implies that $g \circ \theta $ is in $H^2(\bD^d)$. 
where $\sigma $ is a permutation in $S_d$. 
Now putting the explicit expression of $\sigma \in G$ in \eqref{Eq: Phi}, we have
$$ |G| \ \Phi(\boldsymbol{z}) = \sum_{\sigma \in G} f \left( \alpha_{1\sigma} z_{\tilde{\sigma}(1)}, \dots, \alpha_{d\sigma} z_{\tilde{\sigma}(d)} \right) = \sum_{\sigma \in G} \sum_{n \in \mathbb{N}} \frac{1}{n} \alpha_{1\sigma}^n \overline{\alpha}_{2\sigma}^n z_{\tilde{\sigma}(1)}^n \overline{z}_{\tilde{\sigma}(2)}^n.
$$
Note that for the identity element $\operatorname{I}_d$ of $G$, $\tilde{\sigma}$ is the identity permutation and $\alpha_{j\sigma}=1$ for all $j$. Therefore the sum above can be written as

\begin{align}\label{Phi}
|G| \ \Phi(\boldsymbol{z})&= \sum_{n\in \mathbb{N}}^\infty \frac{1}{n} z_1^n \bar{z}_2^n + \sum_{\sigma \in G\setminus \{\operatorname{I}_d\}} \sum_{n \in \mathbb{N}} \frac{1}{n} \alpha_{1\sigma}^n \overline{\alpha}_{2\sigma}^n z_{\tilde{\sigma}(1)}^n \overline{z}_{\tilde{\sigma}(2)}^n
\end{align}
Consider the closed subspace $\mathcal{M}_0$ of $L^2(\bT^d)$ defined by 
$$\mathcal{M}_0 = \left\lbrace \psi \in L^2(\bT^d): \psi (\boldsymbol{z})= \sum_{n \in \mathbb{Z}} a_n z_1^n \bar{z}_2^n \text{ s.t. } \sum_{n\in \mathbb{Z}} |a_n|^2 <\infty \right\rbrace $$
and the orthogonal projection $\operatorname{P}_{\mathcal{M}_0}$  from $L^2(\bT^d)$ onto $\mathcal{M}_0$. Since $\Psi \circ \theta \in L^\infty(\bT^d)$,
\begin{align}\label{eq-1}
\|\Psi \circ \theta \|_\infty = \|\operatorname{M}_{\Psi \circ \theta} \| \geq \| \operatorname{P}_{\mathcal{M}_0} \operatorname{M}_{\Psi \circ \theta} \operatorname{P}_{\mathcal{M}_0} \|.
\end{align}
But $g\circ \theta $ being holomorphic, a simple calculation shows that
$$
\operatorname{P}_{\mathcal{M}_0} \operatorname{M}_{\Psi \circ \theta} \operatorname{P}_{\mathcal{M}_0} = \operatorname{P}_{\mathcal{M}_0} \operatorname{M}_{\Phi + g\circ \theta(0)} \operatorname{P}_{\mathcal{M}_0}. 
$$
Therefore, $\|\Psi \circ \theta \|_\infty \geq \| \operatorname{P}_{\mathcal{M}_0} \operatorname{M}_{\Phi + g\circ \theta(0)} \operatorname{P}_{\mathcal{M}_0} \|$. For $\psi =\sum_{n \in \mathbb{Z}} a_n z_1^n \bar{z}_2^n $ in $\mathcal{M}_0$, define a linear map $V: \mathcal{M}_0 \to L^2(\bT)$ by
$$ V(\psi) = \sum_{n \in \mathbb{Z}} a_n z ^n. $$ Clearly, $V$ is a unitary. Further, for $\psi$ as above we have

$$
\| \operatorname{P}_{\mathcal{M}_0} \operatorname{M}_{\Phi + g\circ \theta(0)} \operatorname{P}_{\mathcal{M}_0}(\psi)\| =  \| \operatorname{M}_{\Phi_0 + g\circ \theta(0)} (V(\psi) )\|
$$
where $\Phi_0 = V(\operatorname{P}_{\mathcal{M}_0} \Phi) $ and $\operatorname{M}_{\Phi_0 + g\circ \theta(0)}$ is the multiplication by $\Phi_0 + g\circ \theta(0)$ on $L^2(\bT)$. So, \eqref{eq-1} implies that $$ \|\Psi \circ \theta \|_\infty \geq \| \operatorname{M}_{\Phi_0 + g\circ \theta(0)} \| $$ and hence $\Phi_0 + g\circ \theta(0)$ must be in $L^\infty(\bT)$. 
From \eqref{Phi} we get
\begin{align*}
\operatorname{P}_{\mathcal{M}_0} \Phi(\boldsymbol{z})=\frac{1}{|G|} \sum_{\substack{\tilde{\sigma}(1)=1,\\ \tilde{\sigma}(2)=2}}  \sum_{n \in \mathbb{N}} \frac{1}{n} \alpha_{1\sigma}^n \overline{\alpha}_{2\sigma}^n z_1^n \overline{z}_2^n + \frac{1}{|G|} \sum_{\substack{\tilde{\sigma}(1)=2,\\ \tilde{\sigma}(2)=1}}  \sum_{n \in \mathbb{N}} \frac{1}{n} \alpha_{1\sigma}^n \overline{\alpha}_{2\sigma}^n z_{2}^n \overline{z}_{1}^n.
\end{align*}
Therefore, 
\begin{align*}
\Phi_0(z)= \frac{1}{|G|} \underbrace{\sum_{n \in \mathbb{N}} \frac{1}{n} \sum_{\substack{\tilde{\sigma}(1)=1,\\ \tilde{\sigma}(2)=2}}   (\alpha_{1\sigma}\overline{\alpha}_{2\sigma})^n z^n }_{\textbf{ sum (i) }}+ \frac{1}{|G|}
 \underbrace{\sum_{n \in \mathbb{N}} \frac{1}{n} \sum_{\substack{\tilde{\sigma}(1)=2,\\ \tilde{\sigma}(2)=1}}   (\alpha_{1\sigma}\overline{\alpha}_{2\sigma})^n \bar{z}^n }_{\textbf{ sum (ii)}}.
\end{align*}
Suppose $\Lambda_1 $ is the subset of $G$ consisting of $\sigma$ appearing in $\textbf{sum (i)}$ and $\Lambda_1'$ consists of those $\sigma \in \Lambda_1$  such that $\alpha_{1\sigma} = \alpha_{2\sigma}$. Similarly, suppose that $\Lambda_2$ is the subset of $G$ consisting of $\sigma$ appearing in $\textbf{sum (ii)}$ and $\Lambda_2'$ consists of those $\sigma$ in $\Lambda_2$ such that $\alpha_{1\sigma} = \alpha_{2\sigma}$. The the expression above can be written as
\begin{align*}
|G|\Phi_0(z) &= \sum_{n \in \mathbb{N}} \frac{|\Lambda_1'|}{n} z^n + \sum_{n \in \mathbb{N}} \frac{1}{n} \sum_{\sigma \in \Lambda_1 \setminus \Lambda_1'}   (\alpha_{1\sigma}\overline{\alpha}_{2\sigma})^n z^n \\
            & + \sum_{n \in \mathbb{N}} \frac{|\Lambda_2'|}{n} \bar{z}^n + \sum_{n \in \mathbb{N}} \frac{1}{n} \sum_{\sigma \in \Lambda_2 \setminus \Lambda_2'}   (\alpha_{1\sigma}\overline{\alpha}_{2\sigma})^n \bar{z}^n 
\end{align*}
Since the identity map $\operatorname{I}_d \in \Lambda_1'$, $|\Lambda_1'| \geq 1$. But $|\Lambda_2'|$ can be zero. Therefore as $z \in \bT$ approaches towards $1$ the first sum blows up to infinity and the same happens for the third sum provided $|\Lambda_2'| \geq 1$. On the other hand, since $\alpha_{1\sigma} \neq \alpha_{2\sigma}$ for each $\sigma$ appearing in the second and the fourth sum, they remains bounded as $z \in \bT$ approaches to $1$. Hence, we conclude that $\Phi_0 $ does not belong to $L^\infty(\bT)$. This contradicts the fact that $\Phi_0 + g\circ \theta \in L^\infty(\bT)$. The proof is complete.
\end{proof}

\part{Projectivity of Hilbert modules}

For the rest of the paper, the domain $\Omega$ will be the polydisc $\bD^d$. The measure $\nu$ is then the normalized Haar measure on $\bT^d$. The pseudo-reflection group will be $G(m,t,d)$ where $m$, $t$ and $d$ are positive integers such that $t$ divides $m$. The elements of this group are unitaries of the form 
$$ (z_1, \ldots , z_d) \mapsto (e^{(2\pi i \nu_1)/m} z_{\sigma(1)}, \ldots , e^{(2\pi i \nu_d)/m} z_{\sigma(d)})$$
where $\sigma$ is a permutation of $\{1, \ldots , d\}$ and $\nu_i$ are integers whose sum is divisible by $t$. Rudin found the basic polynomial maps explicitly for these groups in \cite{Rudin-IUMJ}. Let $E_1, \ldots E_d$ be the elementary symmetric polynomials in $d$ variables. Then,
\begin{equation} \label{RudinThetas}
\theta_i(z_1, \ldots , z_d) = E_i (z_1^m, \ldots , z_d^m) \text{ for } 1 \le i \le d-1 \text{ and } \theta_d(z_1, \ldots , z_d) = (z_1 \ldots z_d)^{m/t}.\end{equation}
  The group $G(m,t,d) $ is a pseudo-reflection group of order $\frac{m^d d!}{t}$.    Note that the choice $(m, t)=(1,1)$ gives $G(1,1,d) = S_d$. 


\section{A non-cramped Hilbert module: A Pisier-type example} \label{Pisier}

Continuing to denote  the coordinates of $\theta(\bD^d)$ by $\boldsymbol{p}=(p_1,p_2\dots, p_d)$ (i.e., $p_j \circ \theta = \theta_j$), we observe that  a Hilbert module $\mathcal{H}$ over $\cA(\theta(\bD^d))$ is contractive if and only if $\theta(\bD^d)$ is a spectral set for the $d$-tuple $\left( T_{p_1},\dots, T_{p_d}\right)$.

While it is obvious that the objects of $\mathfrak{C}$ are objects of $\mathfrak{H}$, the converse is intricate. Indeed, this question for the classical case of $\cA(\bD)$ was open for long time until Pisier \cite{Pisier} settled it in the negative: \textit{there is a polynomially bounded Hilbert space operator which is not similar to a contraction}. 
\begin{thm}\label{CneqH}
The category $\mathfrak{H}$ is strictly larger than the category $\mathfrak{C}$.
\end{thm}

The rest of the section will prove this theorem. 
\begin{definition}\label{D:TheUsualHeroes}
A $d$-tuple of commuting bounded normal operators $(T_1,\dots, T_d)$ on a Hilbert space $\mathcal{H}$ is said to be a $\overline{\theta(\bD^d)}$-unitary provided the Taylor joint-spectrum

\noindent $\sigma_{Taylor} (T_1,\dots, T_d)$ is contained in the \v Silov boundary $\partial \theta(\bD^d)=\theta(\bT^d)$. If a $d$-tuple of commuting bounded operators $(T_1,\dots, T_d)$ is the restriction of a $\overline{\theta(\bD^d)}$-unitary to a joint-invariant subspace then, $(T_1,\dots, T_d)$ is said to be a $\overline{\theta(\bD^d)}$-isometry.

A Hilbert module over $\cA(\theta(\bD^d))$ is said to be \textit{isometric} (or, \textit{unitary}) if $\left( T_{p_1},\dots, T_{p_d}\right)$ is a $\overline{\theta(\bD^d)}$-isometry (or, a $\overline{\theta(\bD^d)}$-unitary). A contractive Hilbert module over $\cA(\theta(\bD^d))$ is said to be {\em pure} if $\left( T_{p_1},\dots, T_{p_d}\right)$ is a $\overline{\theta(\bD^d)}$-contraction and $T_{p_d}$ is a pure contraction. 
\end{definition}
It is easy to verify that the tuple of co-ordinate multipliers $(M_{p_1}, \dots, M_{p_d})$ is a $\overline{\theta(\bD^d)}$-unitary on $L^2(\partial \theta(\bD^d), \mu_{\rho, \theta})$ and the restriction of this tuple to the Hardy space $H_\rho ^2(\theta(\bD^d))$ is a $\overline{\theta(\bD^d)}$-isometry.

\begin{example}The Hilbert modules
$H_\rho ^2(\theta(\bD^d))$ and $L^2(\partial \theta(\bD^d), \mu_{\rho, \theta})$ over $\cA(\theta(\bD^d))$ with the usual action of polynomials in $p_1, \dots, p_d$ are two examples of contractive Hilbert modules. Furthermore, it is also easy to check that the first one is a $\overline{\theta(\bD^d)}$-isometric module and the second one is a $\overline{\theta(\bD^d)}$-unitary module.
\end{example}

\noindent {\em Proof of Theorem \ref{CneqH}.} 
Step 1 recalls the Davidson-Paulsen example because Step 2 uses the finer details of it. 

\textbf{Step 1: A brief recall of the Davidson-Paulsen example \cite{Davidson-Paulsen}:}

The $2 \times 2$ matrices $
V=\begin{bmatrix}
1 & 0 \\
0 & -1
\end{bmatrix},
C=\begin{bmatrix}
0 & 0 \\
1 & 0
\end{bmatrix}, \text{ and }
I_2= \begin{bmatrix}
1 & 0 \\
0 & 1
\end{bmatrix}
$
satisfy 
\begin{align} \label{Pisier-exp-eqn}
V^2= I_2, C^2=0, CV=C, VC=-C, C^*C = \begin{bmatrix}
1 &0 \\
0 &0
\end{bmatrix}=E_{11}, \text{ and } CC^* = \begin{bmatrix}
0 &0\\
0& 1
\end{bmatrix}= E_{22}.
\end{align}
For each $n \in \mathbb{N}$,  the operators,
$$
C_i = V^{\otimes i} \otimes C \otimes I_2^{\otimes (n-i-1)} \text{ on } (\bC^2)^{\otimes i} \otimes \bC^2 \otimes (\bC^2)^{\otimes (n-i-1)} \text{ for } 0 \leq i \leq n-1
$$
 act on $\bC^{2^n}$. Using \eqref{Pisier-exp-eqn}, we get that for $0 \leq i \leq n-1$, 
\begin{enumerate}
\item[(i)] $C_i^2 =0$;
\item[(ii)] $C_i^* C_i = I_2^{\otimes i } \otimes E_{11} \otimes  I_2^{\otimes n-i-1 } ;$
\item[(iii)] $C_i C_i^* = I_2^{\otimes i } \otimes E_{22} \otimes  I_2^{\otimes n-i-1 } ;$
\item[(iv)] $C_i C_j + C_j C_i =0$ and $ C_i C_j^* + C_j ^* C_i =0$ for $0\leq i <j \leq n-1$.
\end{enumerate}
Then we have the following norm inequalities \cite[Section 1]{Davidson-Paulsen}.
\begin{proposition}
For given complex scalars $a_0, \dots, a_{n-1}$, we have 
\begin{align} \label{norm-estimation-C-i}
\frac{1}{2} \sum_{i=0}^{n-1} |a_i| \leq \left\|  \sum_{i=0}^{n-1} a_i C_i \otimes C_i \right\| \leq \sum_{i=0}^{n-1} |a_i|.
\end{align}
\end{proposition}

For each $n\in \mathbb{N}$, we relabel the operators $C_i$ as $C_{i,n}$ for $0 \leq i \leq {n-1}$ and we define $C_{i,n} =0$ for $i \geq n$. The operators $W_i$ defined on $\oplus_{n=1}^{\infty} \bC^{2^n}$ by $W_i = \oplus_{n=1}^{\infty} C_{i,n}$ satisfy
\begin{enumerate}
\item[(i)] $
\|\sum_{i=0}^\infty \alpha_i W_i \|^2 = \sum_{i=0}^\infty |\alpha_i|^2 \text{ for } (\alpha_0, \alpha_1, \dots) \in \ell^2$;
\item[(ii)] $ \frac{1}{2} \sum_{i=0}^{n-1} |a_i| \leq \left\|  \sum_{i=0}^{n-1} a_i C_{i,n} \otimes W_i \right\| \leq \sum_{i=0}^{n-1} |a_i| \,  \text{ for  }a_0, \dots, a_{n-1} \in \mathbb{C}$;
\item[(iii)] $W_i W_j + W_j W_i = 0;$ and 
\item[(iv)] $W_i W_j^* + W_j^* W_i = \delta_{ij} (I - P_i)$ where $P_i$ is the projection from $\oplus_{n=1}^{\infty} \bC^{2^n}$ onto $\oplus_{n=1}^{i} \bC^{2^n}$.
\end{enumerate}
With this preparation, we quote a very spacial case of \cite[Theorem 3.1]{Davidson-Paulsen}. 
\begin{thm}\label{Davidson-Paulsen-exp}
Consider the sequence $\{a_k\}$ where $a_k= (k+1)^{-\frac{3}{2}}$ for $k \geq 0$. Let $X= (a_{i+j} W_{i+j})_{i,j}$ be a Hankel operator on $\ell^2(\cH)$ where $\cH= \oplus_{n=1}^{\infty} \bC^{2^n}$. Consider the operator 
$$
F= \begin{bmatrix}
S^* & X \\
0 & S
\end{bmatrix}
$$
where $S: \ell^2(\cH) \rightarrow \ell^2(\cH)$ such that $S(h_0, h_1, \dots)= (0, h_0, h_1, \dots)$ is the shift operator of multiplicity $dim(\cH)$. Then $F$ is  then $F$ is polynomially bounded and is not similar to a contraction.
\end{thm}

Define a linear map $\delta : \bC[z] \rightarrow \cB(\ell^2(\cH))$ by setting $\delta(z^k)=X_k = kXS^{k-1}$ and extending linearly. Since $X$ is a Hankel matrix, we have $\delta( f)=X f'(S)$ where $f'$ is the derivative of $f$.

\textbf{Step 2: $F^m$ is polynomially bounded but not similar to a contraction for $m \in \mathbb{N}$:}

It is obvious that $F^m$ is polynomially bounded. So, it remains to prove that $F^m$ is not similar to a contraction. We modify the proof of Theorem \ref{Davidson-Paulsen-exp}. If possible, suppose that 
$$
F^m = \begin{bmatrix}
S^{*m} & X_m \\
0 & S^m
\end{bmatrix}
$$
is similar to a contractions, where $X_{m}$ is the Hankel operator $m(a_{i+j+m-1} W_{i+j+m-1})= \delta(z^{m})$ and $X_0=0$. Therefore, $F^m$ is completely polynomially bounded. It is easy to check that $F^m$ is completely bounded implies that $\delta$ is completely bounded on the space of polynomials in $z^m$, $\bC[z^m]= \operatorname{span}_\bC \{1, z^m, z^{2m},\dots\}$. Note that $(0,0)$ entry of $\delta(z^{km+1})= (km+1) a_{km} W_{km}$. Consider the linear map $\delta_{m,0} : \bC[z^m] \rightarrow \cB(\cH)$ such that $\delta_{m,0}(1)=0$, $\delta(z^{km+1})= (km+1) a_{km} W_{km}$ for $k \geq 1$. Since $\delta$ is completely bounded on $\bC[z^m]$, so is $\delta_{m,0}$. 

We shall show that $\delta_{m,0}$ cannot be completely bounded. To that end, for each $n \geq 1$, consider the $2^{mn} \times 2^{mn}$ matrix-valued polynomial in $z^m$, $p(z) = \sum_{k=1}^n km\ \bar{a}_{km-1} C_{km-1, mn} z^{mk}$. By $\delta_{m,0}^{(2^{mn})}$, we denote the map from $M_{2^{mn}}(\bC)\otimes \bC[z^m]$ to $ M_{2^{mn}}(\bC)\otimes \cB(\cH)$ corresponding to $\delta_{m,0}$. Then
\begin{align*}
\delta_{m,0}^{(2^{mn})} (P) = \sum_{k=1}^n (km)^2|\bar{a}_{km-1}|^2 C_{km-1, mn} \otimes W_{km-1}.
\end{align*}
So,
\begin{align*}
\|\delta_{m,0}^{(2^{mn})} (P) \|= \left\| \sum_{i=0}^{mn-1} \alpha_i C_{i, mn} \otimes W_i \right\|
\end{align*}
where $\alpha_{km-1}= (km)^2 |a_{km-1}|^2 $ for $k=1,\dots,n$ and $\alpha_i =0 $ otherwise. Therefore,
\begin{align*}
\|\delta_{m,0}^{(2^{mn})} (P) \| \geq \frac{1}{2} \sum_{i=0}^{mn-1} |\alpha_i|= \frac{1}{2} \sum_{k=1}^n (km)^2|\bar{a}_{km-1}|^2.
\end{align*}
Again, $\|p\|_\infty = \sup_{|z|=1} \|p(z)\|$. Using the properties of $C_i$ one can show that, the map $\Lambda : \bC^{mn} \rightarrow \cB(\cH)$ given by 
\begin{align*}
\Lambda(\alpha_0, \dots, \alpha_{mn-1})= \sum_{i=0}^{mn-1} \alpha_i C_i
\end{align*}
is an isometry. For any $|z|=1$, we choose $\alpha_{km-1}= km \ \bar{a}_{km-1} z^{km} $ for $k=1, \dots, n$ and $\alpha_i =0 $ otherwise. Then,
$$
\|p(z)\|^2 = \|\Lambda(\alpha_0, \dots, \alpha_{mn-1})\|^2 = \sum_{k=1}^n (km)^2 |a_{km-1}|^2.
$$
Hence, 
\begin{align*}
\frac{\|\delta_{m,0}^{(2^{mn})} (p) \|}{\|p\|_\infty} \geq \frac{1}{2} \left( \sum_{k=1}^n (km)^2 |a_{km-1}|^2 \right)^{\frac{1}{2}} = \frac{1}{2} \left( \frac{1}{m} \sum_{k=1}^n \frac{1}{k} \right)^{\frac{1}{2}}.
\end{align*}

Since $\sum_{k=1}^n \frac{1}{k}$ goes to $+\infty$ as $n\to +\infty$, the above inequality shows that the $\delta_{m,0}$ is not completely bounded which is a contradiction. Hence $F^m$ can not be similar to a contraction.

\textbf{Step 3: $\mathfrak C \subsetneq \mathfrak H$:}

Let $T$ be a polynomially bounded operator acting on $\cH$ such that $T^m$ is not similar to a contraction for each $m\in \mathbb{N}$. Define the $d$-tuple of commuting operators, $\underline{T}= (T_1 \dots, T_d)$ by $T_j = \theta_j(T, T\dots, T)$ for every $j$. Let $g \in \bC[\boldsymbol{p}]$  such that $g(\boldsymbol{p})= \sum_{\boldsymbol{\alpha} \in \mathbb{Z}_{+}^d} a_{\boldsymbol{\alpha}} \boldsymbol{p}^{\boldsymbol{\alpha}}$ where $a_{\boldsymbol{\alpha}}$'s are zero except for finitely many $\boldsymbol{\alpha}$. Then

 \begin{align*}
 g(\underline{T}) &= \sum_{\boldsymbol{\alpha} \in \mathbb{Z}_{+}^d} a_{\boldsymbol{\alpha}} \underline{T}^{\boldsymbol{\alpha}} 
                  = \sum_{\boldsymbol{\alpha} \in \mathbb{Z}_{+}^d} a_{\boldsymbol{\alpha}} T_1^{\alpha_1} \dots T_d ^{\alpha_d} = r(T) 
 \end{align*}
where $r(z)=\sum_{\boldsymbol{\alpha} \in \mathbb{Z}_{+}^d}  a_{\boldsymbol{\alpha}} \theta(z,\dots, z)^{\boldsymbol{\alpha}} \in \mathbb{C}[z]$. Since $T$ being polynomially bounded, there is a fixed constant $c>0$ (independent of $r$) such that $\|g(\underline{T})\|=\|r(T)\| \leq c \|r\|_{\infty, \overline{\mathbb{D}}} $. But 
$$ \|r\|_{\infty, \overline{\mathbb{D}}} = \sup_{z\in \overline{\mathbb{D}}} |r(z)| = \sup_{z\in \overline{\mathbb{D}}} |g\circ \theta (z, \dots , z)| \leq \|g\|_{\infty, \overline{\theta(\bD^d)}}. $$
Therefore $\|g(\underline{T}) \| \leq c \|g\|_{\infty, \overline{\theta(\bD^d)} }$. Hence we can consider the $\cA(\theta(\bD^d))$ Hilbert module structure on $\mathcal{H}$ with the module action given by,  
$$g \cdot h = g(\underline{T})h $$
for all $h\in \mathcal{H}$ and $g \in \cA(\theta(\bD^d))$. 
So, $\mathcal{H}$ is an object in the category $\mathfrak{H}$. We show that this is not an object in $\mathfrak{C}$. Towards that end, assume that the Hilbert module $\mathcal{H}$ as above is cramped i.e., there exists $X: \mathcal{H} \to \mathcal{K}$ invertible bounded module map such that $\mathcal{H}$ and $\mathcal{K}$ are similar via $X$ for some contractive Hilbert module $\mathcal{K}$. By our construction, the action of $p_j$ on $\mathcal{H}$ is $T_j$ for $j=1,\dots, d$. Suppose for each $j$, $T_{p_j}$ denote the action of $p_j$ on $\mathcal{K}$. Then $\left(T_{p_1},\dots, T_{p_d} \right)$ is a $\overline{\theta(\bD^d)}$-contraction. Therefore $\|T_{p_1}\|\leq d $. Also by the similarity $X$, we have $T_1= X^{-1} T_{p_1} X$. So, $\| X T_1 X^{-1}\| \leq d$ i.e., $d \|XTX^{-1}\| \leq d$ as $T_1=d T^m$. This implies $T^m$ is similar to a contraction which is a contradiction. Hence the Hilbert module $\mathcal{H}$ is not an object in $\mathfrak{C}$.
\qed

\section{The Hardy space is not a projective Hilbert module}\label{Application}
 
\begin{thm}\label{Thm:NonProj}
Consider the quotient domain $\theta(\bD^d)$ obtained by the action of the pseudo-reflection group $G=G(m,t,d)$. For any $\rho \in \widehat{G}_1$, the Hilbert module $H_\rho ^2(\theta(\bD^d))$ is not projective in $\mathfrak{H}$ as well as in $\mathfrak{C}$.
\end{thm}
The proof will require several intermediate steps. Expectedly, an appropriate notion of Toeplitz operators will pop up. But first we need a suitable orthonormal basis for $R_\rho ^G H^2(\bD^d)$. 

Recall from \eqref{P-rho} the orthogonal projection $\mathbb{P}_\rho : L^2(\bT^d) \to L^2( \bT^d )$ defined as
\begin{align*}
\mathbb{P}_\rho (\psi) = \frac{1}{|G|} \sum_{\sigma \in G} \rho(\sigma^{-1}) \sigma(\psi), \text{ for } \psi \in L^2(\bT^d).
\end{align*}
We denote the independent variables on $\bT^d$ by $\boldsymbol{\zeta} = (\zeta_1, \dots, \zeta_d)$ and their conjugates by $\overline{\boldsymbol{\zeta}}= (\bar{\zeta}_1, \dots, \bar{\zeta}_d)$.
So, the set
\begin{align}\label{TheOrthoSet}
\left\lbrace \mathbb{P}_\rho (\boldsymbol{\zeta}^{\boldsymbol{\alpha}}) : \boldsymbol{\alpha} \in \mathbb Z^d  \right \rbrace
\end{align}
spans $R_\rho ^G L^2(\bT^d) = \mathbb{P}_\rho (L^2(\bT^d))$. Using the structure of the elements of the group $G$ and the orthogonality of $\{\boldsymbol{\zeta}^{\boldsymbol{\alpha}} : \boldsymbol{\alpha} \in \mathbb{Z}^d\}$, one can verify that, if $\boldsymbol{\alpha}= (\alpha_1, \dots, \alpha_d)$ and $\boldsymbol{\beta}=(\beta_1,\dots, \beta_d)$ are such that $\boldsymbol{\alpha} \triangle \boldsymbol{\beta} \neq\emptyset$ (here $\triangle$ denotes the symmetric difference of $\boldsymbol{\alpha}$ and $\boldsymbol{\beta}$ considered as sets), then  
$$\langle \mathbb{P}_\rho (\boldsymbol{\zeta}^{\boldsymbol{\alpha}}), \mathbb{P}_\rho (\boldsymbol{\zeta}^{\boldsymbol{\beta}})\rangle =0.$$
If on the other hand, $\boldsymbol{\alpha} \triangle \boldsymbol{\beta}=\emptyset$, then $\boldsymbol{\beta}$ can be obtained from $\boldsymbol{\alpha}$ by a permutation. In this case, it follows from the definition of $\mathbb P_\rho$ that $\mathbb P_\rho (\boldsymbol{\zeta}^{\boldsymbol{\alpha}})$ is a (unimodular) constant multiple of $\mathbb P_\rho (\boldsymbol{\zeta}^{\boldsymbol{\beta}})$. Thus the set \eqref{TheOrthoSet} with two elements identified when one is a constant multiple of the other, is an orthogonal basis of $R_\rho ^G L^2(\bT^d)$. Each element of this orthogonal basis is a Laurent polynomial. Furthermore $P_{R_\rho ^G H^2(\bT^d)}\mathbb{P}_\rho (\boldsymbol{\zeta}^{\boldsymbol{\alpha}})$ is zero precisely when $\boldsymbol{\alpha}$ has a negative coordinate. This way we arrive at an orthonormal basis $\{ e_\lambda : \lambda \in \Lambda \subset \mathbb Z^d\}$ for $R_\rho ^G L^2(\bT^d)$ such that $\{ e_\lambda : \lambda \in\Lambda_+= \Lambda\cap\mathbb Z_+^d\}$ is an orthonormal basis for $R_\rho ^G H^2(\bT^d)$. Note that one advantage (which will be used) of the basis elements being Laurent polynomials is that given a $\lambda\in\Lambda$, there is an $N$ large enough so that $\theta_d^Ne_\lambda$ belongs to $R_\rho ^G H^2(\bT^d)$. We shall denote the avatar of this basis for the space $L^2(\partial\theta(\bD^d),\mu_{\rho, \theta})$ via the unitary $\Gamma_\rho'$ by $\{t_\lambda:=\Gamma_\rho'^* e_\lambda:\lambda\in\Lambda\}$.


\begin{definition}\label{D:Toeplitz}
Let $m,t,d$ with $t$ dividing $m$,  the group $G = G(m,t,d)$ and the basic polynomial $\theta$ be as above. 
For $\rho \in \widehat{G}_1$ and a symbol $\varphi \in L^2(\partial \theta(\bD^d), \mu_{\rho, \theta})$, define the linear transformation $T_{\varphi}$ on $\bC[\boldsymbol{p}]$, the space of holomorphic polynomials in variables $p_1,p_2,\dots,p_d$ (the coordinate functions in $\theta(\bD^d)$) as
\begin{align*}
T_\varphi (g)= \operatorname{P}_{H_\rho^2(\theta(\bD^d))} (\varphi g), \text{ for all } g \mbox{ in } \bC[\boldsymbol{p}]. 
\end{align*}
If $T_\varphi$ defines a bounded linear operator on $\bC[\boldsymbol{p}]$, then it extends uniquely to a bounded operator on $H_\rho^2(\theta(\bD^d))$. In that case we say that $T_\varphi$ is a bounded \textit{Toeplitz operator}.
\end{definition}
Examples include the coordinate multipliers $T_{p_j}$ on $H_\rho^2 (\theta(\bD^d))$. Note that the tuple $(T_{p_1},T_{p_2}, \dots, T_{p_d})$ is a $\overline{\theta(\bD^d)}$-isometry because it is the restriction of the $\overline{\theta(\bD^d)}$-unitary $(M_{p_1},M_{p_2}, \dots, M_{p_d})$ acting on $L^2 (\partial \theta(\bD^d), \mu_{\rho, \theta})$. It is noteworthy that $T_{p_d}$ and hence $T_{\theta_d}$, (avatar of $T_{p_d}$ on the space $R_\rho^G H^2(\bD^d)$ via the unitary $\Gamma_\rho$) are pure isometries. Indeed, using the form \eqref{RudinThetas} of $\theta_d$ we see that for $f, g \in R_\rho^G H^2(\bD^d)$,
\begin{align*}
\left \langle T_{\theta_d}^{*n} f, g  \right\rangle &= \left \langle f, {\theta_d}^{n} g  \right\rangle = \left \langle f, (z_1\dots z_d)^{\frac{mn}{t}} g  \right\rangle = \left \langle (\overline{z_1\dots z_d})^{\frac{mn}{t}} f,  g  \right\rangle.                                               
\end{align*}Since the operator $T_{(z_1\dots z_d)^{\frac{m}{t}}}$ is a pure isometry on $H^2(\bD^d)$ , $\| T_{\theta_d}^{*n} f\| \rightarrow 0$ as $n\to \infty$. A further consequence of the form of the coordinate function $\theta_d$ is that if $(T_1,\dots, T_d)$ is any $\overline{\theta(\bD^d)}$-unitary acting on Hilbert space $\mathcal{H}$, then $T_d$ is unitary. This is because $T_d$ is a normal operator (by definition) and its spectrum $\sigma(T_d)$ is contained in $\{\theta_d(\boldsymbol{z}): \boldsymbol{z} \in \bT^d \}$. But $\theta_d(\boldsymbol{z})= (z_1\dots z_d)^{\frac{m}{t}} \in \bT$ whenever $\boldsymbol{z}\in \bT^d$. These observations will be used later in this section.

  Inspired by the recent work \cite{BDSIMRN} and the classic \cite{BrownHalmos}, the characterizing property of those symbols $\varphi \in L^2(\partial \theta(\bD^d), \mu_{\rho, \theta})$ that define bounded Toeplitz operators is called the \textit{Brown-Halmos relations}. The proof is omitted because it appeared recently in \cite{GSR}.

\begin{thm} \label{Brown-Halmos}
For a symbol $\varphi$ in $L^2(\partial \theta(\bD^d), \mu_{\rho, \theta})$ if $T_\varphi$ is a bounded Toeplitz operator, then the following relations hold:
\begin{align}\label{B-H-relation}
T_{p_{d-j}}^* T_\varphi T_{p_d}^t = T_\varphi T_{p_j} \text{ for } 1\leq j < d \ \text{ and } T_{p_d}^* T_\varphi T_{p_d} = T_\varphi.
\end{align}
(Here $t$ is the parameter corresponding to the group $G(m,t,d)$.) Conversely, if $T_\varphi$ defines a bounded Toeplitz operator on $H_\rho^2(\theta(\bD^d))$ for a symbol $\varphi$ in $L^2(\partial \theta(\bD^d), \mu_{\rho, \theta})$, then there exists a $\psi \in L^\infty (\partial \theta(\bD^d), \mu_{\rho, \theta})$ such that $T_\varphi = P_{H_\rho^2(\theta(\bD^d))} M_\psi|_{H_\rho^2(\theta(\bD^d))}$.
\end{thm}

Let $X$ and $Y$ be two Hilbert modules over $\mathcal A(\theta(\bD^d))$. We shall denote by $\operatorname{Hom}(X,Y)$ the group of all module maps from $X$ into $Y$. \cref{Brown-Halmos} allows us to identify certain homomorphism groups as described in the results below.

\begin{thm}\label{Hom-lemma1}
Consider the Hilbert modules $H_\rho ^2(\theta(\bD^d))$ and $ L^2(\partial \theta(\bD^d),\mu_{\rho, \theta})$ over $\cA(\theta(\bD^d))$. For every module map 
$T$ from $H_\rho ^2(\theta(\bD^d))$ into $L^2(\partial \theta(\bD^d), \mu_{\rho, \theta})$, there exists a $\psi$ in the algebra $L^\infty \left(\partial \theta(\bD^d),\mu_{\rho, \theta} \right)$ such that $T= M_\psi|_{H_\rho ^2(\theta(\bD^d))}$, i.e., 
$$
\operatorname{Hom}\left(H_\rho ^2(\theta(\bD^d)), L^2(\partial \theta(\bD^d), \mu_{\rho, \theta}) \right) = L^\infty \left(\partial \theta(\bD^d),\mu_{\rho, \theta}\right).
$$
\end{thm}

\begin{proof}
Let $\psi$ in $L^2(\partial \theta(\bD^d), \mu_{\rho, \theta})$ be such that $\psi = T(\mathbbm{1})$ where $\mathbbm{1}$ is the constant $1$ function. Since $T$ is a module map, for every polynomial $g \in \cA(\theta(\bD^d))$, we have $T(g)=\psi g$, and therefore
 \begin{align} \label{rec-eq}
P_{H_\rho^2(\theta(\bD^d))}T (f) = P_{H_\rho^2(\theta(\bD^d))}(\psi f) \quad \forall f \in H_\rho^2(\theta(\bD^d)).
\end{align}
This shows that $\psi$ defines a bounded Toeplitz operator $T_\psi$ and in fact $T_\psi=P_{H_\rho^2(\theta(\bD^d))}T$. Invoke Theorem \cref{Brown-Halmos} to get a $\varphi \in L^\infty (\partial \theta(\bD^d), \mu_{\rho, \theta})$ such that 
\begin{align}\label{KeyEqns}
  P_{H_\rho^2(\theta(\bD^d))}T =P_{H_\rho^2(\theta(\bD^d))}M_{\psi}|_{H_\rho^2(\theta(\bD^d))}= P_{H_\rho^2(\theta(\bD^d))}M_{\varphi}|_{H_\rho^2(\theta(\bD^d))}. 
\end{align} Recall that the orthonormal basis $\{t_\lambda = \Gamma_\rho'^* e_\lambda : \lambda \in \Lambda \}$ for $L^2(\partial \theta(\bD^d), \mu_{\rho, \theta})$. We shall show that $\varphi=\psi$ and $T=M_\psi|_{H_\rho^2(\theta(\bD^d))}$. From \eqref{KeyEqns} we have for every $f\in H_\rho^2(\theta(\bD^d))$ and $\lambda \in \Lambda_+$,
 \begin{align*}
\langle Tf,t_\lambda \rangle=\left \langle \psi f, t_\lambda \right \rangle = \left \langle \varphi f, t_\lambda \right \rangle.
\end{align*}All that remains to do is to establish the same equalities for every $\lambda$ in $\Lambda$. Fix $\lambda \in \Lambda$. There is an $N$ large enough so that $p_d^Nt_\lambda$ is in $H_\rho^2(\theta(\bD^d))$. Thus applying \eqref{KeyEqns} again we see that
 \begin{align*}
\langle Tp_d^Nf,p_d^Nt_\lambda \rangle=\left \langle \psi p_d^Nf, p_d^Nt_\lambda \right \rangle = \left \langle \varphi p_d^Nf, p_d^Nt_\lambda \right \rangle.
\end{align*}Now apply the fact that $T$ is a module map to get $Tp_d^Nf=p_d^N Tf$ and then since $M_{p_d}$ is a unitary on $L^2(\partial \theta(\bD^d),\mu_{\rho, \theta})$, the above equalities are the same as
 \begin{align*}
\langle Tf,t_\lambda \rangle=\left \langle \psi f, t_\lambda \right \rangle = \left \langle \varphi f, t_\lambda \right \rangle.
\end{align*}Here $\lambda\in\Lambda$ and $f\in H_\rho^2(\theta(\bD^d))$ are arbitrary. This not only shows that $\varphi=\psi$ (by taking $f=\mathbbm 1$ in the above conclusion) but also it derives that $T= M_\psi|_{H_\rho ^2(\theta(\bD^d))}$.


Conversely, for any $\psi \in L^\infty \left(\partial \theta(\bD^d), \mu_{\rho, \theta}\right)$, $T= M_{\psi}|_{H_\rho^2(\theta(\bD^d))}$ defines a bounded operator from $H_\rho^2(\theta(\bD^d)) $ to $L^2(\partial \theta(\bD^d),\mu_{\rho, \theta}) $ and clearly it is a module map over $\cA(\theta(\bD^d))$. Hence the map $T\mapsto T(\mathbbm 1)$ defines a group isomorphism from

\noindent $\operatorname{Hom}\left(H_\rho^2(\theta(\bD^d)),L^2(\partial \theta(\bD^d), \mu_{\rho, \theta}) \right)$ onto  $L^\infty \left(\partial \theta(\bD^d),\mu_{\rho, \theta}\right)$.
\end{proof}

For notational simplicity we write $H_\rho^2$ and $L_\rho^2$ to denote $H_\rho^2(\theta(\bD^d))$ and $L^2(\partial \theta(\bD^d), \mu_{\rho, \theta}) $ respectively for the result stated below.

\begin{thm}\label{Hom-lemma2}
The group $\operatorname{Hom}\left(H_\rho^2, L_\rho^2 / H_\rho^2 \right)$ is isometrically isomorphic to $\operatorname{Hank}(\partial \theta(\bD^d))$ (those symbols in ${H_\rho^2}^\perp$ which define bounded big Hankel operators) via the mapping $T\mapsto T(\mathbbm 1)$.
\end{thm}
\begin{proof} 
Let $T\in \operatorname{Hom}\left(H_\rho ^2, L_\rho^2 /H_\rho^2  \right)$. Note that $L_\rho ^2 /H_\rho ^2 \cong (H_\rho ^2 )^\perp$. For any holomorphic polynomial $f$, $T(f)=fT(\mathbbm{1})$ and $T(\mathbbm{1}) \in L_\rho ^2/H_\rho ^2 $ can be uniquely identified with some $\psi \in (H_\rho^2)^\perp$. 
From this we can define, 
$$\Psi :\operatorname{Hom}\left(H_\rho^2 , L_\rho ^2 /H_\rho ^2 \right) \to \operatorname{Hank}(\partial \theta(\bD^d)) \ \text{by}\  \Psi(T)=\psi .$$ Since $T(\mathbbm{1})= \psi + H_\rho^2 $, we have  
$$ f\cdot T(\mathbbm{1}) = \psi f + H_\rho^2 = P_{H_\rho^2}^\perp (\psi f) + H_\rho^2                                       $$
for every holomorphic polynomial $f$ and hence 
$$ \|P_{H_\rho^2}^\perp (\psi f)\|_2 =  \|P_{H_\rho^2}^\perp (\psi f) + H_\rho^2 \|_{\text{quotient norm}} = \|T(f)\| \leq \|T\| \|f \|_2$$
So, $H_{\psi}$ extends as a bounded linear operator from $H_\rho^2$ to $(H_\rho^2)^\perp$ and hence $\psi \in \operatorname{Hank}(\partial \theta(\bD^d))$. This shows the map $\Psi $ is well defined. Also, $\|H_\psi \| \leq \|T\|$. 

Further, for any $\psi \in \operatorname{Hank}(\partial \theta(\bD^d))$ we have $T\in \operatorname{Hom}\left(H_\rho ^2, L_\rho^2 /H_\rho^2  \right)$ defined as
$$ T(f)= \psi f + H_\rho^2 $$
for each polynomial $f\in H_\rho^2$. Note that, 
$$\|T(f)\| = \|P_{H_\rho^2}^\perp (\psi f) + H_\rho^2 \|_{\text{quotient norm}} \leq \|H_\psi\|\|f \| $$
and therefore $\|T\| \leq \|H_\psi\|$. Thus $\Psi$ is an isometric isomorphism from $\operatorname{Hom}\left(H_\rho ^2, L_\rho^2 /H_\rho^2  \right)$ to $\operatorname{Hank}(\partial \theta(\bD^d))$ endowed with the norm $\|\psi\|_{\operatorname{Hank}(\partial \theta(\bD^d))}=\|H_\psi\| $ for each $\psi \in \operatorname{Hank}(\partial \theta(\bD^d))$.
\end{proof}

The extension group can be understood well through the following theorem from \cite{Car-Cla1}.

Suppose $\mathcal{H}$ and $\mathcal{K}$ are $ \cA(\theta(\bD^d))$ Hilbert modules. Let $\mathfrak{A}= \mathfrak{A}(\mathcal{K},\mathcal{H})$ be the space of all {\em co-cycles}, i.e., continuous bilinear functions $\sigma: \cA(\theta(\bD^d))  \times \mathcal{K} \to \mathcal{H}$ which satisfy the condition 
\begin{align}\label{cocycle}
f \cdot\sigma(g,k)+ \sigma(f, g\cdot k)= \sigma(fg,k)
\end{align}
for $f,g\in \cA(\theta(\bD^d))$ and $k\in \mathcal{K}$. Examples of co-cycles include bilinear functions of the form
\begin{align}\label{coboundary}
\sigma(f,k)= f\cdot L(k)- L(f\cdot k)
\end{align} for some bounded linear operator $L:\mathcal{K} \to \mathcal{H}$. In general, a co-cycle may not arise from a bounded linear operator as above; those who do are called {\em co-boundaries}. We denote by $\mathfrak{B}=\mathfrak{B}(\mathcal{K},\mathcal{H})$ the subspace of $\mathfrak A$ consisting of co-boundaries. 
\begin{thm}[See Theorem 2.2.2 in \cite{Car-Cla1}]\label{cohom}
Let $\mathcal{H}$ and $\mathcal{K}$ be $ \cA(\theta(\bD^d))$ Hilbert modules, and $\mathfrak{A}$, $\mathfrak{B}$ be the spaces of co-cycles and co-boundaries, respectively. Then $$ \operatorname{Ext}_{\mathfrak{H}}(\mathcal{K}, \mathcal{H})= \mathfrak{A}/\mathfrak{B}.$$
\end{thm}
See \cite[Theorem 2.2.2]{Car-Cla1} for a proof of \cref{cohom}, where the same conclusion is proved for Hilbert modules over any continuous function algebra $\cA$. Thus, in view of \cref{L:Projective}, a Hilbert module $\mathcal K$ is projective if and only if given any Hilbert module $\cH$, every co-cycle $\sigma:\mathcal A(\theta(\bD^d))\times\cK\to\cH$ is a co-boundary. This will prove useful in the result below, which shows that a unitary Hilbert module over $\cA(\theta(\bD^d))$ is almost a projective Hilbert module in the cramped category. 
\begin{proposition}
Let $\mathcal{K}$ be a $\overline{\theta(\bD^d)}$-unitary Hilbert module over $\cA(\theta(\bD^d))$. Then $$ \operatorname{Ext}_{\mathfrak{C}}\left(\mathcal{K}, \cH\right) = 0$$for every pure cramped Hilbert module $\cH$ over $\cA(\theta(\bD^d))$.
\end{proposition}
\begin{proof}
Consider a co-cycle $\sigma: \cA(\theta(\bD^d)) \times \cK \to \cH $. Let $U_{p_j}$ be the multiplication by the co-ordinate function $p_j$ on $\cK$ for $j=1,\dots,d$. Then by our assumption, $\left( U_{p_1},\dots, U_{p_d} \right) $ is a $\overline{\theta(\bD^d)}$-unitary on $\cK$ (refer to Definition \ref{D:TheUsualHeroes}). Also, we have the $d$-tuple of co-ordinate multipliers $\left(T_{p_1},\dots, T_{p_d} \right) $ on $\cH$ such that $T_{p_d}$ is a pure contraction.
Note that for every $f\in \cK$ and $n\geq1$, \begin{align*}
\|\sigma (p_d ^n, U_{p_d}^{*n}f)\| \leq \|\sigma \| \|p_d ^n\|_{\overline{\theta(\bD^d)},\infty} \|U_{p_d}^{*n}f\| = \|\sigma \| \|f\|.
\end{align*}
So, $\{\sigma (p_d^n, U_{p_d}^{*n}f) \}_n$ is a norm bounded sequence for each $f \in \cK$. Consider a translation invariant Banach limit $\operatorname{LIM}$ (say) on $\ell_\infty (\mathbb{N})$, the space of all bounded sequences of complex numbers. Define a bounded linear operator $L: \cK \rightarrow H$ by 
$$\langle Lf, g \rangle = \operatorname{LIM} \left\lbrace \langle \sigma (p_d^n, U_{p_d}^{*n}f), g \rangle  \right\rbrace_n, $$
for $f \in \cK$ and $g\in \cH$. We claim that
\begin{align} \label{claim1}
(LU_{p_d} - T_{p_d} L)f = \sigma (p_d, f)  \ \text{for}\  f\in \cK. 
\end{align}
Indeed, for $f\in \cK$ and $g \in \cH $
\begin{align*}
&\langle (T_{p_d}L - LU_{p_d})f, g \rangle \\
&=\langle T_{p_d} Lf, g \rangle - \langle L U_{p_d} f, g \rangle \\
&=\operatorname{LIM} \left \lbrace \langle p_d \sigma (p_d^n, U_{p_d}^{*n}f), g \rangle \right \rbrace_n  - \operatorname{LIM} \left \lbrace \langle \sigma(p_d^n, U_{p_d}^{*n}U_{p_d} f), g \rangle \right \rbrace_n \\
&= \operatorname{LIM} \left \lbrace \langle p_d \sigma (p_d^n, U_{p_d}^{*n}f), g \rangle \right \rbrace_n - \operatorname{LIM} \left \lbrace \langle p_d \sigma (p_d^{n-1}, U_{p_d}^{*(n-1)}f), g \rangle - \langle \sigma (p_d, p_d^{n-1} U_{p_d}^{*(n-1)}f), g \rangle \right \rbrace_n \\
&= \operatorname{LIM} \left \lbrace \langle p_d \sigma (p_d^n, U_{p_d}^{*n}f), g \rangle \right \rbrace_n - \operatorname{LIM} \left \lbrace \langle p_d\sigma(p_d^{n-1}, U_{p_d}^{*(n-1)}f), g \rangle \right \rbrace_n - \langle \sigma (p_d , f), g \rangle \\
&= - \langle \sigma (p_d ,f), g \rangle \ (\text{ by translation invariance of}\ \operatorname{LIM}).
\end{align*}
This proves \eqref{claim1}. Now for $h\in \cA(\theta(\bD^d))$, we define a linear operator on $\cK$ by
$$Tf= L(h\cdot f) - h\cdot Lf -\sigma(h, f) \ \text{for}\ f\in \cK.$$
The goal is to show that $T$ is the zero operator. Towards that we compute
\begin{align*}
&(T_{p_d}T- T U_{p_d})f \\
&= T_{p_d}(L(h\cdot f) - h\cdot Lf -\sigma (h, f) ) - L(p_d h\cdot f) + h \cdot L(p_d \cdot f) + \sigma(h, p_d \cdot f) \\
&=h\cdot (L(p_d \cdot f) - p_d \cdot Lf) + (T_{p_d} L -LU_{p_d})(h\cdot f) + \sigma (h, p_d\cdot f) - p_d \cdot \sigma (h,f)\\
&= h \cdot \sigma(p_d,f) - \sigma (p_d,h\cdot f)+ \sigma (h, p_d \cdot f) - p_d \cdot \sigma (h,f) \ \text{ (by \eqref{claim1}) }\\
&= \sigma (p_d \cdot h, f)-\sigma (p_d \cdot h, f) =0.
\end{align*} 
Therefore we have $T_{p_d} T = T U_{p_d}$. This implies that 
$$ U_{p_d}^{*n}T^* = T^* T_{p_d}^{*n}.$$ Now since $T_{p_d}$ is a pure contraction and $U_{p_d}$ is a unitary operator, 
\begin{align*}
\|T^*g\|=\|U_{p_d}^{*n}T^*g \| = \|T^* T_{p_d}^{*n} g\| \leq \|T^*\| \|T_{p_d}^{*n}g\| \rightarrow 0.
\end{align*}
Consequently for every $f\in \cK$ and $h\in \cA(\theta(\bD^d))$,
\begin{align*}
\sigma (h,f)= L(h\cdot f) - h\cdot L(f)
\end{align*}
showing that an arbitrary co-cycle $\sigma $ is a co-boundary, i.e., $\operatorname{Ext}(\cK, \cH)= \{0\}$ whenever $\cH$ is a pure cramped Hilbert module over $\cA(\theta(\bD^d))$.
\end{proof}

 
Before we prove \cref{Thm:NonProj}, we recall a result in homological algebra that shows how one can get a long exact sequence from a short exact sequence. A proof of this can be found in \cite{Car-Cla1}. The statement involves certain maps induced by a Hilbert module map. Let $\cH,\cK$ and $\cE$ be any Hilbert modules over $\cA(\theta(\bD^d))$, and $\alpha : \cH \to \cK$ be a module map. Then $\alpha$ induces a natural map $\alpha_* : \textup{Hom}_{\mathfrak{H}}(\cE, \cH) \to \textup{Hom}_{\mathfrak{H}}(\cE, \cK)$ such that $\alpha_* (T) = \alpha \circ T$ for each $T \in \textup{Hom}_{\mathfrak{H}}(\cE, \cH)$. Moreover, $\alpha$ induces another map which, with a slight abuse of notation, we again denote by $\alpha_*: \textup{Ext}_{\mathfrak{H}}(\cE, \cH) \to \textup{Ext}_{\mathfrak{H}}(\cE, \cK)$ defined in the following way. Let $E: 0 \longrightarrow \mathcal{H} \xlongrightarrow{\beta} \cJ \xlongrightarrow{\gamma} \mathcal{E} \longrightarrow 0 $ be a representative of a equivalence class $[E]$ in $\textup{Ext}_{\mathfrak{H}}(\cE, \cH)$ then a representative $\alpha E$ of $\alpha_*([E])$ is obtained by the diagram below:
 
\[
  \begin{tikzcd}
    E: 0 \arrow{r} & \mathcal{H} \arrow{r}{\beta}  \arrow{d}{\alpha} & \cJ \arrow{r}{\gamma} \arrow[d, dashed,"\psi "]  & \mathcal{E} \arrow{r} \arrow[d, equal] & 0 \\
    \alpha E: 0 \arrow{r} & \mathcal{K} \arrow[r, dashed, "\beta' "] & \cJ' \arrow[r, dashed, "\gamma' "]& \cE \arrow{r} & 0
  \end{tikzcd}
\]
Here, with $\overline W$ as the closure in $\cK\oplus\cJ$ of the subspace $W:=\{(\alpha (h),-\beta(h)):h\in\cH\}$, $\cJ'=(\cK\oplus\cJ)/{\overline W}$, the maps $\beta'$ and $\psi$ are the natural inclusions and $\gamma'$ is the map induced by composition of the projection of $\cK \oplus \cJ$ onto $\cJ$ followed by $\gamma$. What remains to check is that $\cJ'$ is a Hilbert module and that the maps $\psi,\gamma'$ and $\beta'$ are module maps. This would amount to what is referred to as pushout of the top-left corner of the diagram. The details can be found in \cite[Proposition 2.1.4]{Car-Cla1}, where the pushouts and the pullbacks are shown to exist in the category of Hilbert modules over any function algebra.

Furthermore, for $T \in \textup{Hom}_{\mathfrak{H}}(\cE, \cH)$ and $[E] \in \textup{Ext}_\mathfrak{H}(\cK, \cH)$ with $E: 0 \longrightarrow \mathcal{H} \xlongrightarrow{\alpha} \cJ \xlongrightarrow{\beta} \mathcal{K} \longrightarrow 0 $, we define $[ET]$ to be the equivalence class of the short exact sequence $ET$ which is the top row of the diagram:
\[
  \begin{tikzcd}
    ET: 0 \arrow{r} & \mathcal{H} \arrow[r, dashed, "\alpha' "]  \arrow[d, equal] & \cJ' \arrow[r, dashed, "\beta'"] \arrow[d, dashed,"\psi"]  & \mathcal{E} \arrow{r} \arrow{d}{T} & 0 \\
    E: 0 \arrow{r} & \mathcal{H} \arrow{r}{\alpha}  & \cJ \arrow{r}{\beta}   & \mathcal{K} \arrow{r}  & 0 
  \end{tikzcd}
\]
Here $\cJ' = \{(j, e) \in \cJ \oplus \cE : \beta(j)= T(e)\}$ and the maps $\psi, \beta'$ form the pullback of the bottom right corner of the diagram, and $\alpha'(h)= (\alpha(h), 0)$ for each $h \in \cH$. See \cite[Proposition 2.1.4]{Car-Cla1} for more details.


\begin{thm}[See Proposition 2.1.5 of \cite{Car-Cla1}]\label{long-exact}
Let $E: 0 \longrightarrow \mathcal{H} \xlongrightarrow{\alpha} \cJ \xlongrightarrow{\beta} \mathcal{K} \longrightarrow 0 $ be a short exact sequence where $\cH,\cJ$ and $\cK$ are Hilbert modules in $\mathfrak{H}$. Let $\cE$ be an object in $\mathfrak{H}$. Then we have the following exact sequence
\[
  \begin{tikzcd}
     0 \arrow{r} & \operatorname{Hom}_{\mathfrak{H}}(\cE,\mathcal{H}) \arrow{r}{\alpha_{*}} & \operatorname{Hom}_{\mathfrak{H}}(\cE,J) \arrow{r}{\beta_{*}} & \operatorname{Hom}_{\mathfrak{H}}(\cE,\mathcal{K}) \arrow{d}{\delta} \\
 & \operatorname{Ext}_{\mathfrak{H}}(\cE,\mathcal{K}) & \arrow{l}{\beta_*} \operatorname{Ext}_{\mathfrak{H}}(\cE,J) & \arrow{l}{\alpha_*} \operatorname{Ext}_{\mathfrak{H}}(\cE,\mathcal{H})
  \end{tikzcd}
\] 
where $\delta$ is the connecting homomorphism and is given by $\delta(T) = [ET]$ for each $T \in \textup{Hom}_\mathfrak{H}(\cE, \cK)$. 
\end{thm}
 A similar theorem holds for the cramped category $\mathfrak{C}$ as well. 
 

\textbf{Proof of Theorem \ref{Thm:NonProj}.}
%
We prove that the Hardy space $H^2_\rho(\theta(\bD^d))$ is not a projective module in any of the categories $\mathfrak H$ or $\mathfrak{C}$. To make the proof notationally less clutter, we write $H^2_\rho$ and $L^2_\rho$ instead of $H^2_\rho(\theta(\bD^d))$ and $L^2(\partial \theta(\bD^d), \mu_{\rho, \theta})$, respectively.  The strategy is to show that in both the categories $\mathfrak{H}$ and $\mathfrak{C}$, there exists a short exact sequence in $\mathcal{S}(H_\rho^2, H_\rho ^2) $ which does not split. Hence $H_\rho^2 $ is not a projective object in $\mathfrak{C}$ as well as in $\mathfrak{H}$. Since the computation is exactly the same for the two categories, we shall not dwell on the category under consideration.
\begin{thm}\label{Hom-thm}
$\operatorname{Ext}_{\mathfrak{H}}\left(H_\rho ^2(\theta(\bD^d)), H_\rho ^2(\theta(\bD^d)) \right) $ and $\operatorname{Ext}_{\mathfrak{C}}\left(H_\rho ^2(\theta(\bD^d)) , H_\rho ^2(\theta(\bD^d)) \right)$ are both non-zero.
\end{thm}
%
\begin{proof}
Consider the short exact sequence
\[
  \begin{tikzcd}
 0 \arrow{r} & H_\rho^2 \arrow{r}{i} & L_\rho^2 \arrow{r}{\pi} & L_\rho ^2/ H_\rho^2 \arrow{r} & 0,
\end{tikzcd}
\]
where $i$ is the inclusion and $\pi$ is the quotient map. Now invoke \cref{long-exact} to get the following exact sequence:
\[
  \begin{tikzcd}
 \operatorname{Hom}\left(H_\rho ^2, H_\rho^2 \right) \arrow{r}{i_*} & \operatorname{Hom} \left(H_\rho^2, L_\rho ^2 \right) \arrow{r}{\pi_*} & \operatorname{Hom}\left( H_\rho ^2, L_\rho ^2/H_\rho ^2 \right)  \arrow{d}{\delta} \\
 \operatorname{Ext}_{\mathfrak{C}}\left(H_\rho^2 ,L_\rho ^2 / H_\rho^2 \right) & \arrow{l}{i_*} \operatorname{Ext}_{\mathfrak{C}}\left(H_\rho ^2,L_\rho^2 \right) &\arrow{l}{\pi_*} \operatorname{Ext}_{\mathfrak{C}}\left(H_\rho^2, H_\rho^2 \right).    
\end{tikzcd}
\]
If $\operatorname{Ext}_{\mathfrak{C}}\left(H_\rho^2, H_\rho^2 \right)= 0$ then, the exactness of the above sequence implies that 
$$ \pi_*: \operatorname{Hom}\left(H_\rho^2, L_\rho ^2 \right) \to \operatorname{Hom} \left(H_\rho ^2, L_\rho^2/H_\rho^2 \right)$$ is a surjective homomorphism. This contradicts \cref{Failure-Nehari} as follows. Pick $\varphi\in \operatorname{Hank}(\partial \theta(\bD^d))$, the set of those symbols in ${H^2_\rho}^\perp$ that defines a bounded Hankel operator. By \cref{Hom-lemma2}, there exists a homomorphism $T:H_\rho^2\to L^2_\rho/H^2_\rho$ such that $T(\mathbbm{1})=\varphi+ H^2_\rho$. Now we apply surjectivity of $\pi_*$ and   \cref{Hom-lemma1} to get $\psi\in L^\infty (\partial \theta(\bD^d), \mu_{\rho, \theta})$ such that $\pi\circ M_\psi|_{H^2_\rho}=\pi_*(M_\psi|_{H^2_\rho})=T$. Applying this equality to the constant function $\mathbbm 1$ we have $\psi+ H^2_\rho=\varphi+H^2_\rho$, which is same as saying $\varphi-\psi \in H^2_\rho$ or equivalently $H_\varphi=H_\psi$. This contradicts \cref{Failure-Nehari}.
\end{proof}


\section{The Normal Category}\label{normal_cat}
\subsection{Normal Hilbert modules}
 We introduce  the category of {\em normal} Hilbert $\cA(\theta(\bD^d))$-modules where we shall find a plenty of projective objects. Our motivation arises from \cite{Guo-Studia}. 
\begin{definition}
A Hilbert $\cA(\theta(\bD^d))$-module $\cH$ is said to be normal if for each $h\in \cH$, the action of $\cA(\theta(\bD^d))$ on $h$,

$$ \cA(\theta(\bD^d)) \ni f \mapsto f\cdot h \in \cH $$
is ($wk^* - wk$)-continuous, i.e., continuous from the weak$^*$ topology of $L^\infty(\partial \theta(\bD^d), \mu_{\rho, \theta})$ restricted to $\cA(\theta(\bD^d))$ to the weak topology on the Hilbert space $\cH$.
\end{definition}
The {\em normal category}, denoted by $\mathfrak{N}$, consists of normal Hilbert $\cA(\theta(\bD^d))$-modules as objects and Hilbert module maps as morphisms. We denote the category of all normal Hilbert modules over the algebra $H^\infty(\theta(\bD^d))$  by $\mathfrak{N}_{\infty}$.

For normal Hilbert modules $\cN_1$ and $\cN_2$, the  extension group $\operatorname{Ext_{\mathfrak{N}}}(\cN_1, \cN_2)$ is defined in a way similar to what was done for the category $\mathfrak{H}$ . A  characterization  of $\operatorname{Ext_{\mathfrak{N}}}(N_1, N_2)$ akin to \cref{cohom} holds. To be more precise, 
$$\operatorname{Ext_{\mathfrak{N}}}(\cN_1, \cN_2) = \mathfrak{A}_{\mathfrak{N}} /\mathfrak{B}_{\mathfrak{N}} $$ 
where $\mathfrak{A}_{\mathfrak{N}}$ is the set of all co-cycles $\sigma : \cA(\theta(\bD^d)) \times \cN_1 \rightarrow \cN_2$ such that for every $h \in N_1$,
the map $ f \mapsto \sigma(f, h)$ is ($wk^* - wk$)-continuous from $\cA(\theta(\bD^d))$  to $\cN_2$ and $\mathfrak{B}_{\mathfrak{N}}$ is the collection of co-boundaries in $\mathfrak{A}_{\mathfrak{N}}$. 

\begin{example}
The Hilbert $\cA(\theta(\bD^d))$-module $L_\rho^2(\partial \theta(\bD^d))$ is a prime example of a normal Hilbert module. To see that, let $h \in L_\rho^2(\partial \theta(\bD^d))$. Suppose $\{f_\lambda \}$ is a net in $\cA(\theta(\bD^d))$ converging to $f$ in the in weak$^*$ topology. Then for every $\psi \in L_\rho^1(\partial \theta(\bD^d))$,
\begin{align}\label{weak}
\int_{\partial \theta(\bD^d)} f_\lambda \psi d\mu_{\rho, \theta} \ \rightarrow \int_{\partial \theta(\bD^d)} f \psi d\mu_{\rho, \theta}.
\end{align}
Now, for $g \in L_\rho^2(\partial \theta(\bD^d))$, by \eqref{weak}, we have
\begin{align*}
\langle f_\lambda h , g \rangle = \int_{\partial \theta(\bD^d)} f_\lambda h \bar{g} d\mu_{\rho, \theta} \ \rightarrow \int_{\partial \theta(\bD^d)} fh \bar{g} d\mu_{\rho, \theta} \ = \langle fh, g \rangle.
\end{align*}
\end{example}
\subsection{The Carath\'eodory approximation}
A {\em rational inner function} on the domain $\theta(\bD^d)$ is a rational function which has its poles off $\overline{\theta(\bD^d)}$ and is continuous as well as unimodular on the boundary. More general inner functions have been considered in the literature. However, we can prove the following approximation theorem with the nice rational inner functions as described above. 
\begin{thm}[Carath\'eodory approximation] \label{Cara-approx}
Let $f$ be in $H^\infty(\theta(\bD^d))$ with $\|f\|_\infty \leq 1$. Then $f$ can be approximated uniformly on compacta in $\theta(\bD^d)$ by a sequence of rational inner functions in $\cA(\theta(\bD^d))$.
\end{thm}
For the proof, see \cite{BK}. In brief, it is enough to approximate the $H^\infty(\bD^d)$ function $f \circ \theta$ by $G$-invariant rational inner functions in $\cA(\bD^d)$.
Towards that end, we use the well-known Carath\'eodory approximation for $\bD^d$, \cite[Theorem 5.5.1]{Rudin} and choose the polynomial $P$ and the monomial $M$ (of sufficiently large degree) as in the proof of Theorem 5.5.1 of \cite{Rudin} so that $P$ and $M$ are $G$-invariant. Then the proof is immediate. See \cite{BK} for the structure of rational inner functions in $\cA(\theta(\bD^d))$. 

Let $\{t_\lambda:=\Gamma_\rho'^* e_\lambda:\lambda\in\Lambda\}$ be the orthonormal basis for $L^2(\partial \theta(\bD^d), \mu_{\rho, \theta})$
as in the discussion preceding Definition \ref{D:Toeplitz}. Let us consider the following subspace of $H^\infty (\theta(\bD^d))$
$$
\widetilde{H^\infty (\theta(\bD^d))}:=\{\varphi\in L^\infty(\partial\theta(\bD^d), \mu_{\rho, \theta}): \langle \varphi,t_\lambda\rangle=0 \mbox{ for }\lambda\in \Lambda\setminus\Lambda_+\}.
$$ We pause to note the following result.
\begin{lemma}\label{L:HRadial}
For every $\varphi$ in $\widetilde{H^\infty (\theta(\bD^d))}$, there exists a $\psi$ in $H^\infty(\theta(\bD^d))$ such that
\begin{align}\label{HRadial}
\lim_{r\to 1-}\psi\circ\theta(r\bm\zeta)=\varphi\circ\theta(\bm\zeta)\mbox{ for } \bm\zeta\in\bT^d \mbox{ $\nu$ almost everywhere}
\end{align}and $\|\varphi\|_{L^\infty(\partial\theta(\bD^d), \mu_{\rho, \theta})}=\|\psi\|_{H^\infty(\theta(\bD^d))}$.
\end{lemma}
\begin{proof}
Given $\varphi$ as in the statement, we first show that 
\begin{align}\label{NormEqual}
\|\varphi\|_{L^\infty(\partial\theta(\bD^d), \mu_{\rho, \theta})} = \|\varphi \circ \theta \|_{L^\infty(\bT^d, \nu)}.
\end{align}That $\varphi\circ\theta$ is measurable follows from the fact that $\mu_{\rho,\theta}$ is the push-forward measure of $|\ell_\rho|^2d\nu$ and that the class of measurable functions with respect to the two measures $|\ell_\rho|^2d\nu$ and $\nu$ on $\bT^d$ are the same. For the norm equality, let $M = \|\varphi\|_{L^\infty(\partial\theta(\bD^d), \mu_{\rho, \theta})}$ and $E \subset \partial\theta(\bD^d)$ such that $|\varphi| \leq M$ on $\partial \theta(\bD^d) \setminus E$. Then $\mu_{\rho, \theta}(E)=0$, i.e., $(|\ell_\rho|^2 d\nu) (\theta^{-1}(E))=0$. Since $\nu(\mathcal{Z}(\ell_\rho)\cap\bT^d)=0$, the measures $\nu$ and $|\ell_\rho|^2d\nu$ are mutually absolutely continuous and therefore $\nu(\theta^{-1}(E)) = 0$ as well.
This implies that $\|\varphi \circ \theta\|_{L^\infty(\bT^d, \nu)}  \leq M$. Suppose $\|\varphi \circ \theta\|_{L^\infty(\bT^d, \nu)} < M$. Then the set
$$
V'=\{\boldsymbol{\zeta}\in\bT^d:\|\varphi \circ \theta\|_{L^\infty(\bT^d, \nu)}<|\varphi \circ \theta(\boldsymbol{\zeta})| \leq M \}
$$ has $\nu$ measure zero, equivalently, $(|\ell_\rho|^2 d\nu)(V')=0$.  Clearly, $V'$ is a $G$-invariant set and so, $\theta^{-1} (\theta(V')) = V'$. Thus $\mu_{\rho, \theta}(\theta(V'))=0$. In other words,
$|\varphi| \leq \|\varphi \circ \theta\|_{L^\infty(\bT^d, \nu)} < M$ on $\partial \theta(\bD^d)$ except for a $\mu_{\rho, \theta}$ measure zero set. This is a contradiction as $M = \|\varphi\|_{L^\infty(\partial\theta(\bD^d), \mu_{\rho, \theta})}$.

Now to obtain the analytic function as in the statement, we proceed as follows. We apply the Poisson extension to the $L^\infty(\bT^d,\nu)$ function 
\begin{align*}
\varphi \circ \theta = \sum_{\lambda \in \Lambda_+} \left \langle \varphi, t_\lambda \right \rangle t_\lambda \circ \theta \text{ on } \bT^d
\end{align*}to get an analytic function on $\bD^d$ as follows
\begin{align} \label{Eq: Poisson extn}
\cP_{\bD^d}[\varphi \circ \theta ](\boldsymbol{z})&= \sum_{\lambda \in \Lambda_+} \left \langle \varphi, t_\lambda \right \rangle t_\lambda \circ \theta (\boldsymbol{z}).
\end{align}
This is possible, see \cite[Chapter 2]{Rudin}. It is easy to check that $\cP_{\bD^d}[\varphi \circ \theta]$ is a $G$-invariant function in $H^\infty(\bD^d)$ and hence by the analytic version of the Chevalley-Shephard-Todd theorem, there exists $\psi \in H^\infty(\theta(\bD^d))$ such that
$$
\psi \circ \theta(\boldsymbol{z})=\cP_{\bD^d}[\varphi \circ \theta](\boldsymbol{z}) =\sum_{\lambda \in \Lambda_+} \left \langle \varphi, t_\lambda \right \rangle t_\lambda \circ \theta (\boldsymbol{z}) \text{ for } \boldsymbol{z}\in \bD^d.
$$
Therefore the radial limit of the $G$-invariant $H^\infty(\bD^d)$ function $\psi \circ \theta $ is 
\begin{align}\label{Eq:Function}
(\psi \circ \theta )^* = \sum_{\lambda \in \Lambda_+} \left \langle \varphi, t_\lambda \right \rangle t_\lambda \circ \theta = \varphi \circ \theta.
\end{align}
This proves \eqref{HRadial}. As for the norm equality as stated in Lemma \ref{L:HRadial} we note that
$$
\|\varphi\|_{L^\infty(\partial\theta(\bD^d), \mu_{\rho, \theta})} = \|\varphi \circ \theta \|_{L^\infty(\bT^d, \nu)} = \|(\psi \circ \theta )^*\|_{L^\infty(\bT^d, \nu)} = \|\psi \circ \theta\|_{H^\infty(\bD^d)} = \|\psi\|_{H^\infty(\theta(\bD^d))}
$$
by using \eqref{NormEqual} and \eqref{Eq:Function}.
\end{proof}
The Carath\'eodory approximation leads to some results related to the Banach algebra $H^\infty(\theta(\bD^d))$ which will be useful for the proof of the main result of this section. In the following, we use the weak$^*$ topology on $H^\infty(\theta(\bD^d))$ as induced from $L^\infty(\partial \theta (\mathbb D^d), \mu_{\rho, \theta})$ by virtue of \cref{Radial}. 
\begin{proposition} \label{wk-density} \leavevmode
\begin{enumerate}
\item[(i)] The closed norm-unit ball $(\cA(\theta(\bD^d)))_1$ of $\cA(\theta(\bD^d))$  is weak$^*$ dense in the closed norm-unit ball $(H^\infty(\theta(\bD^d)))_1$ of $H^\infty(\theta(\bD^d))$.
\item[(ii)] The subspace $H^\infty (\theta(\bD^d))$ is weak$^*$ closed in $L^\infty(\partial \theta(\bD^d), \mu_{\rho, \theta})$.
\item[(iii)] The closed ball $(H^\infty(\theta(\bD^d)))_1$ is weak$^*$ compact and weak$^*$ metrizable.
\end{enumerate}
\end{proposition}
\begin{proof} 

(i) Let $\varphi \in (H^\infty(\theta(\bD^d)))_1$. Apply \cref{Cara-approx}, to get sequence $\{\tilde{g}_j\}$ of rational inner functions in $\cA(\theta(\bD^d))$ such that $\tilde{g}_j$ converges to $\varphi$ uniformly on compacta in $\theta(\bD^d)$. Then $g_j = \tilde{g}_j \circ \theta $ converges to $\varphi \circ \theta$ uniformly on compacta in $\bD^d$.
For any polynomial $\psi$ in $\boldsymbol{\zeta}$ and $\overline{\boldsymbol{\zeta}}$, the Cauchy integral formula yields that 
\begin{align*}
\int_{\bT^d} (\widetilde{g_j} \circ \theta ) \psi d\nu \rightarrow \int_{\bT^d} (\varphi \circ \theta ) \psi d\nu.
\end{align*}
By density of the polynomials in $\boldsymbol{\zeta}$ and $\overline{\boldsymbol{\zeta}}$ in $L^1(\bT^d, \nu)$, for any $\psi \in L^1(\bT^d, \nu)$ we have
 \begin{align}\label{L1-convergence}
\int_{\bT^d} (\widetilde{g_j} \circ \theta) \psi d\nu \rightarrow \int_{\bT^d} (\varphi \circ \theta) \psi d\nu.
\end{align}
Therefore for any $\psi \in L^1(\partial\theta(\bD^d), \mu_{\rho, \theta}) $,
\begin{align*}
\int_{\partial \theta(\bD^d)} \widetilde{g_j} \psi d\mu_{\rho, \theta} &= \int_{\bT^d} (\widetilde{g_j} \circ \theta ) (\psi \circ \theta ) |\ell_\rho|^2 d\nu \\
               & \rightarrow \int_{\bT^d} (\varphi \circ \theta) (\psi \circ \theta) |\ell_\rho|^2 d\nu \ (\text{by}\ \eqref{L1-convergence}) = \int_{\partial \theta(\bD^d)} \varphi \psi d\mu_{\rho, \theta}.
\end{align*}

So, $\widetilde{g_j} \xrightarrow{\text{weak}^*} \varphi $ such that $\widetilde{g_j} \in \cA(\theta(\bD^d))$. This proves the weak$^*$ density of $(\cA(\theta(\bD^d)))_1$ in $(H^\infty(\theta(\bD^d)))_1$.

(ii) To prove the second part, let us consider a sequence $\{\varphi_j\}$ in $H^\infty(\theta(\bD^d))$ such that $\varphi_j \rightarrow \varphi $ in weak$^*$ topology.  Here Lemma \ref{Radial} is used to view the functions $\varphi_j$ as members of $L^\infty(\partial\theta(\bD^d),\mu_{\rho,\theta})$.  The weak$^*$ convergence therefore means that
\begin{align}\label{WTM}
\int_{\partial\theta(\bD^d)} \varphi_j \psi d\mu_{\rho, \theta} \to \int_{\partial\theta(\bD^d)} \varphi \psi d\mu_{\rho, \theta}\quad\mbox{ for every } \psi \in L^1(\partial\theta(\bD^d),\mu_{\rho,\theta}).
\end{align} This ensures that the $L^\infty(\partial\theta(\bD^d),\mu_{\rho,\theta})$ function $\varphi$ has not negative Fourier coefficients when considered as an $L^2(\partial\theta(\bD^d),\mu_{\rho,\theta})$ function. Indeed, recall the orthonormal basis $\{t_\lambda:=\Gamma_\rho'^* e_\lambda:\lambda\in\Lambda\}$ for $L^2(\partial \theta(\bD^d), \mu_{\rho, \theta})$
from the discussion preceding Definition \ref{D:Toeplitz}. Using \eqref{WTM} we note that for every $\lambda\in\Lambda$,
\begin{align*}
\left \langle \varphi, t_\lambda \right \rangle = \int_{\partial \theta(\bD^d)} \varphi \overline{t_\lambda} d\mu_{\rho, \theta} 
                                             = \lim_{j \to \infty} \int_{\partial \theta(\bD^d)} \varphi_j \overline{t_{\lambda}} d\mu_{\rho, \theta} = \lim_{j\to \infty} \langle  \varphi_j , t_\lambda \rangle.
\end{align*}
Since $\langle\varphi_j,t_\lambda\rangle=0$ if $\lambda\in\Lambda\setminus\Lambda_+$, the same is true for $\varphi$ as the computation above shows.


%
%
%
%
%
 Therefore, $\varphi \in \widetilde{H^\infty (\theta(\bD^d))}$. Hence, by \cref{L:HRadial}, we see that $\varphi$ is radial limit of some $\psi \in H^\infty(\theta(\bD^d))$ with $ \|\psi\|_{H^\infty(\theta(\bD^d))} =\|\varphi\|_{L^\infty(\partial \theta(\bD^d), \mu_{\rho, \theta})}$.
 Thus $H^\infty(\theta(\bD^d))$ is weak$^*$ closed.

(iii) To prove the last part, we note that $(H^\infty(\theta(\bD^d)))_1$ is weak$^*$ closed subset of the unit ball of $L^\infty(\partial \theta(\bD^d), \mu_{\rho, \theta})$ which is weak$^*$ compact and hence $(H^\infty(\theta(\bD^d)))_1$ is weak$^*$ compact. A general result in the theory of Banach spaces says that a Banach space $X$ is separable if and only if the closed norm-unit ball of $X^*$ is metrizable. Therefore, $L^1(\partial \theta(\bD^d), \mu_{\rho, \theta})$ being separable, the unit ball of $L^\infty(\partial \theta(\bD^d), \mu_{\rho, \theta})$ is weak$^*$ metrizable and hence so is $(H^\infty(\theta(\bD^d)))_1$.
\end{proof}
The proof of part (i) actually provides us a deeper implication.
\begin{proposition}\label{inner-density}
The rational inner functions in $\cA(\theta(\bD^d))$ generate the Banach algebra $H^\infty(\theta(\bD^d))$ in the restriction of the weak$^*$ topology of $L^\infty(\partial \theta(\bD^d), \mu_{\rho, \theta})$.
\end{proposition}
\subsection{The projective objects}
 We produce projective objects via the following theorem related to the extension group.
\begin{thm} \label{NormalExt}
 Suppose $\cN$ is also a Hilbert module over $C(\partial \theta(\bD^d))$. If $\cN$, as a  Hilbert module over $\cA(\theta(\bD^d))$, is normal then for every normal Hilbert module $\cK$ over $\cA(\theta(\bD^d))$ we have
$$\operatorname{Ext}_{\mathfrak{N}} (\cK, \cN)= 0 \text{ and } \operatorname{Ext}_{\mathfrak{N}} (\cN, \cK)= 0 .$$
\end{thm}

\begin{proof}
The weak$^*$ density of the unit ball of $\cA(\theta(\bD^d))$ in the unit ball of $H^\infty(\theta(\bD^d))$ (by Proposition \ref{inner-density}) implies that that the module action of $\cA(\theta(\bD^d))$ on a normal Hilbert module can be uniquely extended to $H^\infty(\theta(\bD^d))$ without increasing the module bound. Therefore, it is enough to prove that 
$$\operatorname{Ext}_{{\mathfrak{N}}_\infty} (\cK,\cN)= 0 \text{ and } \operatorname{Ext}_{{\mathfrak{N}}_\infty} (\cN, \cK)= 0,$$
where $\cN$ and $\cK$ are regarded as Hilbert modules over the algebra $H^\infty(\theta(\bD^d))$. To prove the theorem, take a co-cycle $\eta : H^\infty(\theta(\bD^d))\times \cK \rightarrow \cN $ in $\mathfrak{A}_{\mathfrak{N}}$. We need to show that $\eta$ is a co-boundary, i.e., we need to produce a bounded linear operator $T:\cK\to\cN$ such that $\eta(\psi,\cdot)=T_\psi T-TT_\psi$ for every $\psi$ in $H^\infty(\theta(\bD^d))$. Let $\cB_1(\cN, \cK)$ be the class of trace class operators from $\cN$ to $\cK$ and $\cB(\cK, \cN)$  be the set of bounded linear operators from $\cK$ to $\cN$. To produce the desired operator $T$, we shall use the fact that $ \cB(\cK, \cN) \equiv \left( \cB_1(\cN, \cK) \right)^* $ via the map $T \mapsto L_T$ where $L_T$ is the continuous linear functional given by
$$ L_T(C)=\operatorname{tr}(TC) \text{ for }C \in \cB_1(\cN, \cK) .$$ 

Consider the Abelian semi-group $\mathfrak{I}(\theta(\bD^d))$ of inner functions on $\theta(\bD^d)$, i.e., functions that are bounded analytic in $\theta(\bD^d)$ with their boundary values on $\partial \theta(\bD^d)$ unimodular almost everywhere with respect to $\mu_{\rho, \theta}$. Let $B(\mathfrak{I}(\theta(\bD^d)))$ denote the space of all complex-valued bounded functions on $\mathfrak{I}(\theta(\bD^d))$. For $\beta \in \mathfrak{I}(\theta(\bD^d))$ and $f\in B(\mathfrak{I}(\theta(\bD^d)))$, we shall denote by $f_\beta$ the bounded function on $\mathfrak{I}(\theta(\bD^d))$ defined by
$$
f_\beta:\psi\mapsto f(\beta\psi).
$$ We shall use what is referred to as the {\em invariant means} of $B(\mathfrak{I}(\theta(\bD^d)))$. These are bounded linear functionals on $B(\mathfrak{I}(\theta(\bD^d)))$ such that 
$$
\mathbb{M}(f_\beta) = \mathbb{M}(f)\quad\mbox{for every }f\in B(\mathfrak{I}(\theta(\bD^d))) \mbox{ and }\beta \in \mathfrak{I}(\theta(\bD^d)).
$$ 
The existence of invariant means is known from \cite{Day}. Define a linear functional $L$ on $\cB_1(\cN, \cK)$ by
$$
L(C):=\mathbb{M} \left(\psi\mapsto \operatorname{tr}( T_{\bar{\psi}}\ \eta(\psi, \cdot) C )  \right)=\mathbb M\big(\psi\mapsto\operatorname{tr}(T_{\bar{\psi}}\ \eta(\psi, \cdot)C)\big).
$$The linearity and boundedness of the trace function and $\mathbb M$ imply the boundedness of $L$ defined as above. And therefore there is a $T$ in $\cB(\cK,N)$ such that
\begin{align}\label{L}
L(C)=\operatorname{tr}(TC)=\mathbb M\big(\psi\mapsto\operatorname{tr}(T_{\bar{\psi}}\ \eta(\psi, \cdot)C)\big).
\end{align}
 The invariance of $\mathbb M$ will be used in what follows. Note that in \eqref{L} (or in the displayed equation above it), the operator $T_\psi$ is the module operator on $\cN$. In the following computations, however, we shall have use of $T_\psi$ as a module operator on both the Hilbert modules $\cN$ and $\cK$. To increase readability, we use a superscript to distinguish these two instances. For $\xi$ in $\mathfrak{I}(\theta(\bD^d))$, we carry out the following computation where we use the linearity and the invariance under commutation of the trace function.
 \begin{align*}
\operatorname{tr} ( (T_{\xi}^{(\cN)} T - T T_{\xi}^{(\cK)})C)=
 \operatorname{tr} (T_{\xi}^{(\cN)} T C) -  \operatorname{tr} ( T T_{\xi}^{(\cK)} C) = \operatorname{tr}( TC T_{\xi}^{(\cN)}) - \operatorname{tr}(T T_{\xi}^{(\cK)}C).
\end{align*}Now we apply \eqref{L} to get the right hand side equal to
\begin{align}\label{SecondEqn}
\notag&\mathbb{M}\big(\psi\mapsto \operatorname{tr}( T_{\bar{\psi}}^{(\cN)}\eta(\psi, \cdot) CT_{\xi}^{(\cN)})\big) -\mathbb{M}\big(\psi\mapsto \operatorname{tr}( T_{\bar{\psi}}^{(\cN)}\eta(\psi, \cdot)T_{\xi}^{(\cK)}C)\big)\\
= & \; \mathbb{M}\big(\psi\mapsto \operatorname{tr}( T_{\bar{\psi} \xi}^{(\cN)} \eta(\psi, \cdot) C)\big)- \mathbb{M}\big(\psi\mapsto \operatorname{tr}( T_{\bar{\psi}}^{(\cN)}\eta(\psi, \cdot)T_{\xi}^{(\cK)} C )\big),
\end{align} 
where to obtain the equality we again used the invariance of the trace function under commutation. Since $\eta$ is a co-cycle, it has the distribution property 
$$
\eta(\psi \xi, \cdot) = T_{\psi}^{(\cN)} \eta(\xi, \cdot) + \eta(\psi, \cdot) T_{\xi}^{(\cK)}
$$ for every $\psi,\xi$ in $\mathfrak{I}(\theta(\bD^d))$. Upon multiplying on the left of the distribution property by the adjoint of $T_\psi^{(\cN)}$ and rearranging we get
\begin{align*}
T_{\bar{\psi}}^{(\cN)} \eta(\psi, \cdot) T_{\xi}^{(\cK)} = T_{\bar{\psi}}^{(\cN)} \eta(\psi \xi, \cdot) - \eta( \xi, \cdot).
\end{align*}
We plug this in the second term of \eqref{SecondEqn} to get
\begin{align}\label{FinCon}
\notag &\operatorname{tr}(T_{\xi}^{(\cN)} T - T T_{\xi}^{(\cK)} C) \\ \notag
&= \mathbb{M}\big(\psi\mapsto \operatorname{tr}( T_{\bar{\psi} \xi}^{(\cN)} \eta(\psi, \cdot) C)\big)- \mathbb{M}\big(\psi\mapsto \operatorname{tr}(T_{\bar{\psi}}^{(\cN)} \eta(\psi \xi, \cdot) C )\big) +  \mathbb{M}\big(\psi\mapsto \operatorname{tr}( \eta( \xi, \cdot) C )\big) \\
&= \operatorname{tr}(\eta( \xi, \cdot) C ).
\end{align}
The second equality in the above computation is obtained using the translation invariance of $\mathbb{M}$. Indeed, let $f\in B(\mathfrak{I}(\theta(\bD^d)))$ be given by $ f(\psi)= \operatorname{tr}( T_{\bar{\psi} \xi}^{(\cN)} \eta(\psi, \cdot) C)$ where $\xi\in \mathfrak{I}(\theta(\bD^d))$ is fixed. Then note that
$$
f_\xi(\psi)=f(\xi\psi)=\operatorname{tr}( T_{\overline{\xi\psi} \xi}^{(\cN)} \eta(\xi\psi, \cdot) C).
$$Since $\mathbb M(f)=\mathbb M(f_\xi)$, we have
$$
\mathbb{M}\big(\psi\mapsto \operatorname{tr}( T_{\bar{\psi} \xi}^{(\cN)} \eta(\psi, \cdot) C)\big)= \mathbb{M}\big(\psi\mapsto \operatorname{tr}(T_{\bar{\psi}}^{(\cN)} \eta(\psi \xi, \cdot) C )\big).
$$ From \eqref{FinCon} we therefore conclude 
$$
T_{\xi}^{(\cN)} T - T T_{\xi}^{(\cK)} = \eta(\xi, \cdot)
$$ for every inner function $\xi$. The density theorem in \cref{inner-density} ensures the same for every $\xi \in H^\infty(\theta(\bD^d))$. This means $\eta$ is a co-boundary in $\mathfrak{B}_{\mathfrak{N}}$ and hence $\operatorname{Ext}_{\mathfrak{N}} (\cK, \cN)= 0.$

Let $\overline{\mathfrak{N}}$ be the category of normal Hilbert modules over $\overline{\cA(\theta(\bD^d))}$, the set of complex conjugates of functions in $\cA(\theta(\bD^d))$. We define $\cN_*$ to be a object in $\overline{\mathfrak{N}}$ with the module action $\bar{f} \cdot h = T_f^*h$ for $f \in \cA(\theta(\bD^d))$ and $h\in \cN$. Similarly, we define $\cK_*$. Since $\cN$ is a normal Hilbert module over $C(\partial \theta(\bD^d))$, so is $\cN_*$.

Again, by a duality argument of the Category theory we can see that $\operatorname{Ext}_{\mathfrak{N}} (\cN, \cK) \equiv \operatorname{Ext}_{\overline{\mathfrak{N}}} (\cK_*, \cN_*)$ as groups. Therefore, $\operatorname{Ext}_{\overline{\mathfrak{N}}} (\cK_*, \cN_*) =0$ and hence $\operatorname{Ext}_{\mathfrak{N}} (\cN, \cK) = 0$. 
\end{proof}
\begin{corollary}
Every normal Hilbert module over $C(\partial \theta(\bD^d))$ is a projective object in the category of normal Hilbert modules over $\cA(\theta(\bD^d))$.
\end{corollary}

\vspace*{5mm}

\noindent \textbf{Funding:} \\
	{\footnotesize{This research is supported by the J C Bose Fellowship  JCB/2021/000041 of SERB, the Prime Minister's Research Fellowship PM/MHRD-21-1274.03 and the DST FIST program-2021 [TPN-700661].}}

\end{document}